\documentclass[onefignum,onetabnum]{siamart190516}

\usepackage{lipsum}
\usepackage{amsfonts}
\usepackage{booktabs,ctable,threeparttable}
\usepackage{etex}
\usepackage{graphicx}
\usepackage{epstopdf}
\usepackage{algorithmic}
\usepackage{bm}
\usepackage{epsfig,picins,picinpar,subfigure}

\ifpdf
  \DeclareGraphicsExtensions{.eps,.pdf,.png,.jpg}
\else
  \DeclareGraphicsExtensions{.eps}
\fi

\numberwithin{equation}{section}

\newtheorem{thm}{Theorem}[section]

\newtheorem{lem}[thm]{Lemma}

\theoremstyle{definition}

\newcommand{\norm}[1]{\left\Vert#1\right\Vert}
\newcommand{\abs}[1]{\left\vert#1\right\vert}

\newcommand{\diag}[1]{\mathrm{diag}(#1)}
\newcommand{\var}[1]{\mathrm{Var}\left(#1\right)}

\newcommand{\cov}[1]{\mathrm{Cov}(#1)}

\newcommand{\mrd}{\mathrm{d}}
\newcommand{\mbe}{\mathbb{E}}
\newcommand{\vech}{\mathrm{vech}}

\newcommand{\sobs}{s^*}
\newcommand{\yobs}{y^*}
\newcommand{\pabc}{\tilde{p}}
\newcommand{\kl}{\mathrm{KL}}
\newcommand{\mbp}{\mathbb{P}}
\newcommand{\mbr}{\mathbb{R}}
\newcommand{\mbn}{\mathbb{N}}
\newcommand{\hk}{V_{\mathrm{HK}}}

\newsiamremark{remark}{Remark}
\newsiamremark{hypothesis}{Hypothesis}
\crefname{hypothesis}{Hypothesis}{Hypotheses}
\newsiamthm{claim}{Claim}

\headers{Unbiased MLMC-based VB for likelihood-free inference}{Z. He, Z. Xu, and X. Wang}

\title{Unbiased MLMC-based variational Bayes for likelihood-free inference\thanks{Submitted to the editors DATE.
\funding{This work of the first author was funded by the National Science Foundation of China (No. 12071154), Guangdong Basic and Applied Basic Research Foundation (No. 2021A1515010275), Guangzhou Science and Technology Program (No. 202102020407). And the third author was funded by the National Science Foundation of China (No. 720711119).}}}

\author{Zhijian He\thanks{School of Mathematics, South China University of Technology, Guangzhou 510641, People's Republic of China (\email{hezhijian@scut.edu.cn}).}
\and Zhenghang Xu\thanks{Corresponding author. Department of Mathematical Sciences, Tsinghua University, Beijing 100084, People's Republic of China (\email{xzh17@mails.tsinghua.edu.cn}).}
\and Xiaoqun Wang\thanks{Department of Mathematical Sciences, Tsinghua University, Beijing 100084, People's Republic of China (\email{wangxiaoqun@mail.tsinghua.edu.cn}).}}

\usepackage{amsopn}
\ifpdf
\hypersetup{
  pdftitle={UMLMC-based variational Bayes for likelihood-free inference},
  pdfauthor={Z. He, Z. Xu , and X. Wang}
}
\fi

\begin{document}

\maketitle

% REQUIRED
\begin{abstract}
Variational Bayes (VB) is a popular tool for Bayesian inference in statistical modeling. Recently, some VB algorithms are proposed to handle intractable likelihoods with applications such as approximate Bayesian computation.  In this paper, we propose several unbiased estimators based on multilevel Monte Carlo (MLMC) for the gradient of Kullback-Leibler divergence between the posterior distribution and the variational distribution when the likelihood is intractable, but can be estimated unbiasedly. The new VB algorithm differs from the VB algorithms in the literature which usually render biased gradient estimators. Moreover, we incorporate randomized quasi-Monte Carlo (RQMC) sampling within the MLMC-based gradient estimators, which was known to provide a favorable rate of convergence in numerical integration. Theoretical guarantees for RQMC are provided in this new setting.
Numerical experiments show that using RQMC in MLMC greatly speeds up the VB algorithm, and finds a better parameter value than some existing competitors do.

\end{abstract}

% REQUIRED

% REQUIRED

\section{Introduction}
In this article, we are interested in variational Bayes (VB), which is widely used as a computationally effective method for approximating the posterior distribution of a Bayesian problem. Let $\yobs$ be the observed data and $\theta\in\mathbb{R}^p$ be the parameter of interest. The posterior distribution $p(\theta|\yobs)\propto p(\theta)p(\yobs|\theta)$, where $p(\theta)$ is the prior and $p(\yobs|\theta)$ is the likelihood function. VB approximates the posterior by a tractable distribution $q(\theta)$ within certain distribution families, chosen to minimize the Kullback-Leibler (KL) divergence between the VB distribution $q(\theta)$ and the posterior $p(\theta|\yobs)$. The optimization problem is usually solved by using the stochastic gradient decent (SGD) algorithm \cite{DJ:2018}. It calls for computing  the gradient of the KL divergence. A difficulty with SGD is that plain Monte Carlo (MC) sampling to estimate the gradient can be error prone or inefficient. Some variance reduction methods have been adopted to improve SGD \cite{MF:2017,PBJ:2012}. On the other hand, randomized quasi-Monte Carlo (RQMC) methods have been used to improve SGD in the VB setting \cite{BWM:2018}. Recently, Liu and Owen \cite{LO:2021} combined RQMC with a second order limited memory method known as L-BFGS for VB.  RQMC methods such as scrambled digital nets proposed by  \cite{Owen1995} were known to provide a favorable rate of convergence in numerical integration \cite{owen1997a}. Improved sampling accuracy translates directly to improved
optimization as shown in \cite{BWM:2018,LO:2021}.

A second difficulty with SGD is due to the absence of the likelihood function $p(\yobs|\theta)$.
In many applications, the likelihood function is intractable making it difficult to render an unbiased gradient estimator of the KL divergence. For example, the likelihood is an intractable high-dimensional integral over the state variables governed by a Markov process in state space-space models \cite{durbin:2012}. More examples can be found in the context of approximate Bayesian computation (ABC). ABC methods provide a way of approximating the posterior $p(\theta|\yobs)$ when the likelihood function is difficult to compute but it is possible to simulate data from the model \cite{peters2012,tavare1997}.

Likelihood-free inference is an active area in Bayesian computation. There are some progresses on using VB in the likelihood-free context. Barthelm{\'e} and Chopin \cite{BC:2014} used a
variational approximation algorithm known as expectation propagation in approximating ABC posteriors.  Tran et al. \cite{tran:2017} developed a new VB with intractable likelihood (VBIL) method, which can be applied to commonly used statistical models without requiring an analytical solution to model-based expectations. Ong et al. \cite{ong:2018} modified the VBIL method to work with unbiased log-likelihood estimates in the synthetic likelihood framework, resulting in the VB  synthetic likelihood (VBSL) method.

We focus on the problems in which the likelihoods are formulated as an intractable expectation. The  KL divergence turns out to be a nested expectation and so does its gradient. It is natural to use nested simulation for estimating these quantities. However, the plain nested estimator is biased. It is critical to develop unbiased gradient estimators for stochastic gradient-based optimization algorithms. To this end, we use the unbiased multilevel Monte Carlo (MLMC) proposed by \cite{Rhee2015} in the framework of  nested simulation. MLMC is a sophisticated variance reduction technique introduced by \cite{Hein1998} for parametric integration and by \cite{Giles2008} for the estimation of the expectations arising from stochastic differential equations. Nowadays MLMC methods have been extended extensively. For a thorough review of MLMC methods, we refer to \cite{Giles2015}.
Nested simulation combined with the MLMC method has been widely studied in the literature due to its broad applicability \cite{Bujok:2015,giles:2018b,GG2019,GHI:2020}.

In this paper, we develop an unbiased nested MLMC-based VB method to deal with intractable likelihoods. Our work is related to \cite{goda:2020}, who developed an unbiased MLMC stochastic gradient-based optimization method for Bayesian experimental designs.
Our proposed VB algorithm finds a better parameter value and a larger evidence lower bounded (ELBO) thanks to unbiased  gradient and ELBO estimators. This leads to a better estimate of the marginal likelihood $p(\yobs)$ compared to the VBIL method, which is an important factor in model selection. We also incorporate the RQMC sampling within the gradient and the ELBO estimators, which reduces the computational complexity effectively. Goda et al. \cite{goda:2020} worked on the MC sampling rather than RQMC. We provide some numerical analysis for both MC and RQMC settings.

The rest of this paper is organized as follows. In \Cref{se:VBIL}, we review some VB methods with intractable likelihoods, such as VBIL and VBSL, and illuminate their limitations. In \Cref{se:UMLMC}, we provide our unbiased MLMC methods for VB and discuss two different estimators of gradient, which are the score function gradient and re-parameterization gradient. In \Cref{Gaussian}, we provide the details of our algorithms when using Gaussian variational family in VB. In \Cref{se:RQMC}, we improve the algorithms by incorporating RQMC and do some numerical analysis. Finally, in \Cref{sec:num}, some numerical experiments are conducted to support the advantages of our proposed methods. Section~\ref{sec:concl} concludes this paper.

\section{Variational Bayes with an intractable likelihood}\label{se:VBIL}
Recall that our target is to estimate the posterior distribution
\begin{equation}
p(\theta|\yobs)= \frac{p(\theta)p(\yobs|\theta)}{p(\yobs)},\label{eq:post}
\end{equation}
where $p(\yobs) = \int p(\theta)p(\yobs|\theta) \mrd \theta$ is usually an  unknown constant (called the marginal likelihood or evidence).
In many applications such as state-space models and ABC, the likelihood is analytically intractable.
For these cases, the likelihood $p(\yobs|\theta)$ is usally formulated as an expectation
\begin{equation}
p(\yobs|\theta)=\mbe[f(x;\yobs)|\theta],\label{eq:intrlik}
\end{equation}
where $x\sim p(x|\theta)$ is the latent variable.

Suppose that there exists an unbiased estimator $\hat{p}_N(\yobs|\theta)$ for the intractable likelihood $p(\yobs|\theta)$ for given $\theta$, where $N$ is an algorithmic parameter relating to the precision in estimating the likelihood.  For estimating \cref{eq:intrlik}, one can take the sample-mean estimator
\begin{equation}
\hat{p}_N(\yobs|\theta) = \frac 1N\sum_{i=1}^N f(x_i;\yobs) ,\label{eq:samplemean}
\end{equation}
where $x_i$ are iid copies of $x$ for a given $\theta$.
In this paper, we restrict our attention to the sample-mean estimator \cref{eq:samplemean}. We should note that for the state-space models, the likelihood can be unbiasedly estimated by an importance sampling estimator \cite{DK:1997}, or by a particle filter estimator \cite{PSGK:2012}. The later case does not fit into our framework.

VB approximates the posterior distribution $p(\theta|\yobs)$ by a tractable density $q_\lambda(\theta)$ with a variational parameter $\lambda$, chosen to minimize the KL divergence from $q_\lambda(\theta)$ to $p(\theta|\yobs)$, which is defined by
\begin{equation*}
\kl(\lambda) = \kl(q_\lambda(\theta)||p(\theta|\yobs)) = \mathbb{E}_{q_\lambda(\theta)}[\log q_\lambda(\theta)-\log p(\theta|\yobs)].\label{eq:kl}
\end{equation*}
Using \cref{eq:post}, we have
$$\log p(\yobs) = \kl(\lambda)+ L(\lambda),$$
where $L(\lambda)$ is defined by
\begin{align*}
L(\lambda) &= \mathbb{E}_{q_\lambda(\theta)}[\log p(\yobs|\theta)+ \log p(\theta)-\log q_\lambda(\theta)].
\end{align*}
Since $\kl(\lambda)\ge 0$, $L(\lambda)$ is a lower bound of the log-evidence $\log p(\yobs)$, which is called the  ELBO. The minimization
of KL is translated to the maximization of the ELBO since the marginal likelihood $p(\yobs)$ is fixed. The problem turns out to solve
$$\lambda^* = \arg\max_{\lambda\in\Lambda} L(\lambda),$$
where $\Lambda$ is the feasible region of $\lambda$.
Stochastic gradient method and its variants are widely used to solve such a problem.
They use a sequence of steps $$\lambda^{(t+1)}= \lambda^{(t)}+\rho_t \nabla_\lambda L(\lambda^{(t)}),$$ where
$\nabla_\lambda L(\lambda)$ is  gradient of the ELBO and $\rho_t>0$ is the learning rate satisfying the Robbins-Monro
conditions: $\sum_{t=0}^\infty \rho_t=\infty$ and $\sum_{t=0}^\infty \rho_t^2<\infty$.
A simple choice is $\rho_t = a/(t+b)$ for some constants $a,b>0$. Some adaptive methods for choosing the learning rate $\rho_t$ were proposed in the literature, notably AdaGrad \cite{DHS:2011} and Adam \cite{KB:2014}.

The key in stochastic gradient methods is to estimate the gradient $\nabla_\lambda L(\lambda)$ unbiasedly. In the literature, the re-parameterization (RP) trick \cite{KW:2013} and the score function (SF) are two popular methods to derive unbiased gradient estimators. Allowing the interchange of differentiation and expectation as required in the SF method, we have
\begin{align*}
\nabla_\lambda L(\lambda) &=\nabla_\lambda \mathbb{E}_{q_\lambda(\theta)}[\log p(\yobs|\theta)+ \log p(\theta)-\log q_\lambda(\theta)]\notag\\
&= \mathbb{E}_{q_\lambda(\theta)}[\nabla_\lambda \log q_\lambda(\theta) (\log p(\yobs|\theta)+ \log p(\theta)-\log q_\lambda(\theta))],
\end{align*}
where we used the fact that $\mathbb{E}_{q_\lambda(\theta)}[\nabla_\lambda \log q_\lambda(\theta)]=0$. If the likelihood function $p(\yobs|\theta)$ is known, it is straightforward to derive an unbiased estimator for $\nabla_\lambda L(\lambda)$ by sampling  $\theta\sim q_\lambda(\theta)$ repeatedly. However, in our setting, $\log p(\yobs|\theta)$ is intractable. The question is how to use the unbiased estimator $\hat{p}_N(\yobs|\theta)$ of the likelihood to construct an  unbiased  SF estimator for $\nabla_\lambda L(\lambda)$.

On the other hand, for applying the RP trick, we assume that there exists a transformation $\theta = \Gamma(\bm u;\lambda)\sim q_\lambda(\theta)$, where the random variate $\bm u\sim p_1(\bm u)$ independently of $\lambda$. Allowing the interchange of differentiation and expectation again, we have
\begin{align}
\nabla_\lambda L(\lambda) &=\nabla_\lambda \mathbb{E}_{q_\lambda(\theta)}[\log p(\yobs|\theta)+ \log p(\theta)-\log q_\lambda(\theta)]\notag\\
&=\nabla_\lambda \mathbb{E}_{\bm u}[\log p(\yobs|\theta)+ \log p(\theta)-\log q_\lambda(\theta)]\notag\\
&=\mathbb{E}_{\bm u}[\nabla_\lambda\Gamma(\bm u;\lambda)\cdot (\nabla_\theta\log p(\yobs|\theta)+ \nabla_\theta\log p(\theta)-\nabla_\theta\log q_\lambda(\theta))],\label{eq:rep}
\end{align}
where $\nabla_\lambda\Gamma(\bm u;\lambda)$ is the Jacobian matrix with entries $[\nabla_\lambda\Gamma(\bm u;\lambda)]_{ij}=\partial\Gamma_j(\bm u;\lambda)/\partial\lambda_i$.
The RP gradient is much complicated than the SF gradient. In \cref{eq:rep}, one needs to estimate the intractable gradient of log-likelihood $\nabla_\theta\log p(\yobs|\theta)$ unbiasedly. Due to the absence of likelihood, the  SF  and RP methods for the traditional VB cannot be applied directly.

The VBIL method proposed by \cite{tran:2017} works with the augmented space $(\theta,z)$, where $z=\log \hat{p}_N(\yobs|\theta)-\log p(\theta|\yobs)$. Let $g_N(z|\theta)$ be the distribution of $z$ given $\theta$. Tran et al. \cite{tran:2017} applied the variational inference for the target distribution
$$p_N(\theta,z) = p(\theta|\yobs)\exp(z)g_N(z|\theta)$$
with a family of distributions of the form $q_\lambda(\theta,z)= q_\lambda(\theta)g_N(z|\theta)$.
The KL divergence in the augmented space is
$$\widetilde{\kl}(\lambda) = \kl(q_\lambda(\theta,z)||p_N(\theta,z)) = \mathbb{E}_{q_\lambda(\theta,z)}[\log q_\lambda(\theta)-\log p(\theta|\yobs)-z].$$
The ELBO in the augmented space is
\begin{align}
\tilde L(\lambda) &= \mathbb{E}_{q_\lambda(\theta,z)}[\log \hat p_N(\yobs|\theta)+ \log p(\theta)-\log q_\lambda(\theta)]\label{eq:elbo_aug}\\
&=L(\lambda) + \mathbb{E}_{q_\lambda(\theta,z)}[z].\notag
\end{align}
Note that
$$\mbe[z|\theta] = \mbe[\log \hat p_N(\yobs|\theta)]-\log p(\yobs|\theta)\le \log\mbe[\hat p_N(\yobs|\theta)]-\log p(\yobs|\theta)= 0$$
by using Jensen's inequality. As a result, $L(\lambda)\ge \tilde{L}(\lambda)$. The equality holds if and only if $\hat p_N(\yobs|\theta)$ is a constant with probability 1 (w.p.1).
Generally, the maximization of $\tilde L(\lambda)$ is not the same as the maximization of $L(\lambda)$ unless $\mathbb{E}_{q_\lambda(\theta,z)}[z]$ is independent of $\lambda$. Tran et al. \cite{tran:2017} made an attempt to choose $N$ as a function of $\theta$ such that $\mathbb{E}[z|\theta]\equiv\tau$ does not depend on $\theta$. By doing so, $\mathbb{E}_{q_\lambda(\theta,z)}[z]=\tau$ does not depend on $\lambda$. Hence, in practice, one needs to adapt $N$ so that the variance of the log-likelihood estimator is approximately constant with $\theta$. Ong et al. \cite{ong:2018} suggested to set some minimum value $N'$ for the initially estimating the likelihood. Then, if some target value for the log-likelihood variance is exceed based on an empirical estimate, an additional number of samples is repeatedly simulated until the target accuracy is achieved. Although the two ELBOs have the same maximizer, there is a gap (i.e., $\tau$) between the maximums of the two ELBOs. The smaller the target accuracy is, the more work is required in estimating the likelihood. Actually, $L(\lambda)$ is a locally marginalized version of $\tilde{L}(\lambda)$, which is tighter. This can help to approximate the evidence better. Furthermore, this tighter lower bound can potentially help to compute the criterion for model selection such as perplexity used in topic modeling.

In fact, if we use an unbiased estimator of $\log p(\yobs|\theta)$ to replace  $\log \hat p_N(\yobs|\theta)$ in \cref{eq:elbo_aug}, then the resulting ELBO corresponds to the original ELBO $L(\lambda)$. However, an unbiased estimator  of $\log p(\yobs|\theta)$  is not trivial.  To overcome this, \cite{ong:2018} proposed to use a synthetic likelihood. Suppose we have a summary statistic $\mathcal{S}=\mathcal{S}(\yobs)$ of dimension $d\ge p$ and the inference is based on the observed value $s$ of the summary statistic $\mathcal{S}$, which is thought to be informative about $\theta$.  Assume that the statistic $\mathcal{S}$ is exactly Gaussian conditional on each value of $\theta$, that is $p(s|\theta)=\phi(s;\mu(\theta),\Sigma(\theta))$, where $\phi$ is the density of multivariate normal with $\mu(\theta)=\mbe[\mathcal{S}|\theta]$ and $\Sigma(\theta) = \mathrm{Cov}(\mathcal{S}|\theta)$. Now the posterior density is given by
$$
p(\theta|s) \propto p(\theta)p(s|\theta) = p(\theta)\phi(s;\mu(\theta),\Sigma(\theta)).
$$
For a given $\theta$, we may simulate summary statistics $\mathcal{S}_1,\dots,\mathcal{S}_N$ under the model given $\theta$. The mean vector $\mu(\theta)$ and the covariance matrix $\Sigma(\theta)$ are then estimated by
\begin{align*}
\hat{\mu}(\theta)&=\frac 1 N \sum_{i=1}^{N} \mathcal{S}_i,\\
\hat\Sigma(\theta)&=\frac 1 {N-1} \sum_{i=1}^N (\mathcal{S}_i-\hat{\mu}(\theta))(\mathcal{S}_i-\hat{\mu}(\theta))^\top,
\end{align*}
respectively. Then an unbiased estimate of the log-synthetic likelihood $\log p(s|\theta)$ is given by
\begin{align*}
\hat\ell_N (s|\theta) = &-\frac{d}{2}\log(2\pi) - \frac{1}{2}\left\lbrace\log \abs{\hat\Sigma(\theta)}+d\log \left(\frac{N-1}{2}\right)-\sum_{i=1}^d\psi\left(\frac{N-i}{2}\right)\right\rbrace\\
&-\frac{1}{2}\left\lbrace\frac{N-d-2}{N-1}(s-\hat\Sigma(\theta))^\top\hat\Sigma(\theta)^{-1}(s-\hat\Sigma(\theta))-\frac{d}{N} \right\rbrace,
\end{align*}
where $\psi(t)= \Gamma'(t)/\Gamma(t)$ denotes the digamma function and $N>d+2$.
By replacing $\log \hat p_N(\yobs|\theta)$ with $\hat\ell_N (s|\theta)$, then $\tilde{L}(\lambda)=L(\lambda)$. However, it should be noted that the unbiasedness of $\hat\ell_N (s|\theta)$ relies heavily on the assumption of the normality of $\mathcal{S}|\theta$, and the inference is based on the information of the summary statistic $s$ rather than the full data $\yobs$.

\section{Unbiased MLMC for variational Bayes}\label{se:UMLMC}
To fix our idea, we work on the likelihood \cref{eq:intrlik} with an unbiased estimate \cref{eq:samplemean}. Now the ELBO is a nested expectation
$$
L(\lambda) = \mathbb{E}_{q_\lambda(\theta)}[\log \mbe[f(x;\yobs)|\theta]+ \log p(\theta)-\log q_\lambda(\theta)].
$$

\subsection{Score function gradient}
Applying the SF method, we reformulate the gradient as
\begin{equation*}
\nabla_\lambda L(\lambda) = \mathbb{E}_{q_\lambda(\theta)}[\nabla_\lambda \log q_\lambda(\theta) (\log \mbe[f(x;\yobs)|\theta]+ \log p(\theta)-\log q_\lambda(\theta))],
\end{equation*}
which is a nested expectation. Define
\begin{equation}
\mathrm{SF}_{N}(\lambda) = \nabla_\lambda \log q_\lambda(\theta) [\log \hat{p}_N(\yobs|\theta)+ \log p(\theta)-\log q_\lambda(\theta)],\label{eq:nested}
\end{equation}
where $\hat{p}_N(\yobs|\theta)$ is given by \cref{eq:samplemean}, and $(\theta,x)\sim q_\lambda(\theta)p(x|\theta)$. Although $\hat{p}_N(\yobs|\theta)$ is an unbiased likelihood estimator, $\mathrm{SF}_{N}(\lambda)$ is generally biased for estimating the gradient $\nabla_\lambda L(\lambda)$. We next show how to find an unbiased estimator for the log-likelihood  by using unbiased MLMC. Let $\psi_{\theta,N}=\log \hat{p}_N(\yobs|\theta)$. It is clear that
$$\lim_{N\to\infty}\mbe[\psi_{\theta,N}|\theta]=\log p(\yobs|\theta).$$
Consider an increasing sequence $0<M_0<M_1<\cdots$ such that $M_\ell\to \infty$ as $\ell\to \infty$.
Then the following telescoping sum holds,
$$\log p(\yobs|\theta) = \mbe[ \psi_{\theta,M_0}|\theta]+\sum_{\ell= 1}^\infty\mbe[ \psi_{\theta,M_\ell}- \psi_{\theta,M_{\ell-1}}|\theta].$$
More generally, if we have a sequence of correction random variables $\Delta \psi_{\theta,\ell}$, $\ell\ge 0$ such that $\mbe[\Delta \psi_{\theta,0}|\theta]=\mbe[\psi_{\theta,M_0}|\theta]$ and for $\ell>0$,
$$\mbe[\Delta \psi_{\theta,\ell}|\theta]=\mbe[\psi_{\theta,M_\ell}-\psi_{\theta,M_{\ell-1}}|\theta],$$
then it follows that
$$\log p(\yobs|\theta) = \sum_{\ell= 0}^\infty \mbe[\Delta \psi_{\theta,\ell}|\theta].$$
Let $w_\ell>0$ satisfying $\sum_{\ell=0}^\infty w_\ell =1$, and let $I$ be an independent discrete random variable with $\mbp(I=\ell) = w_\ell$. We then have
$$\log p(\yobs|\theta)  =  \mbe\left[\frac{\Delta \psi_{\theta,I}}{w_I}\bigg|\theta\right].$$
Define
\begin{equation}
\mathrm{SF}_{\text{MLMC}}(\lambda) = \nabla_\lambda \log q_\lambda(\theta) \left[\frac{\Delta \psi_{\theta,I}}{w_I}+ \log p(\theta)-\log q_\lambda(\theta)\right],
\end{equation}\label{eq:SF}
which is unbiased for the gradient $\nabla_\lambda L(\lambda)$.
For any number of outer samples $S\ge 1$, the following gradient estimator,
\begin{equation}
\widehat{\nabla_\lambda L}^{\mathrm{SF}}(\lambda) =\frac{1}{S} \sum_{i=1}^S \mathrm{SF}_{\text{MLMC}}^{(i)}(\lambda),\label{eq:SF_esti}
\end{equation}
is unbiased, where $\mathrm{SF}_{\text{MLMC}}^{(i)}(\lambda)$ are iid copy of $\mathrm{SF}_{\text{MLMC}}(\lambda)$ for the MC sampling.

Now
$$\psi_{\theta,M_\ell} =\log \hat{p}_{M_\ell}(\yobs|\theta)= \log \left(\frac 1 {M_\ell}\sum_{i=1}^{M_\ell} f(x_i;\yobs)\right),$$
where $x_i\sim p(x|\theta)$ independently.
We take $\Delta \psi_{\theta,0}=\psi_{\theta,M_0}$. For $\ell\ge 1$, we take an antithetic coupling estimator $$\Delta \psi_{\theta,\ell}=\psi_{\theta,M_\ell}-\frac{1}{2}\left(\psi_{\theta,M_{\ell-1}}^{(a)}+\psi_{\theta,M_{\ell-1}}^{(b)}\right),$$ where
$$\psi_{\theta,M_{\ell-1}}^{(a)} = \log \left(\frac 1 {M_{\ell-1}}\sum_{i=1}^{M_{\ell-1}} f(x_i;\yobs)\right),\ \psi_{\theta,M_{\ell-1}}^{(b)} = \log \left(\frac 1 {M_{\ell-1}}\sum_{i=M_{\ell-1}+1}^{M_\ell} f(x_i;\yobs)\right).$$
The strategy of antithetic coupling is widely used in the MLMC literature \cite{GZ:2014,GHI:2020}, which yields a better rate of convergence for smooth functions. Denote $C_\ell$ as the expected cost of computing $\Delta \psi_{\theta,\ell}$, which is proportional to $M_\ell$. To ensure a finite variance and finite expected computational cost of $\mathrm{SF}_{\text{MLMC}}(\lambda)$, it is required that
\begin{equation}\label{eq:finiteCondition}
\sum_{\ell=0}^\infty \frac{\mbe[\Delta \psi_{\theta,\ell}^2||\nabla_\lambda \log q_\lambda(\theta)||_2^2]}{w_\ell}<\infty\text{ and }\sum_{\ell=0}^{\infty}C_\ell w_\ell<\infty.
\end{equation}
In this paper, we take $M_\ell=M_02^\ell$ for some $M_0\ge 1$ and all $\ell\ge 0$, implying $C_\ell =O(2^\ell)$. Assume that $\mbe[\Delta \psi_{\theta,\ell}^2||\nabla_\lambda \log q_\lambda(\theta)||_2^2] = O(2^{-r\ell})$ for some $r>1$. Let $w_\ell = w_0 2^{-\alpha\ell}$ for $w_0=1-2^{-\alpha}$ and $\alpha>0$.  Then \cref{eq:finiteCondition} holds if we take $\alpha\in (1,r)$. The expected computational cost is then proportional to
\begin{equation}
C(\alpha,M_0)= \sum_{\ell=0}^\infty M_\ell w_\ell = \sum_{\ell=0}^\infty M_0w_02^{(1-\alpha)\ell}=\left(1+\frac{1}{2^{\alpha}-2}\right)M_0. \label{eq:cost}
\end{equation}

%The variance is bounded by
%$$V(\alpha,r)=V_0\sum_{\ell=0}^\infty 2^{-r\ell} w_02^{\alpha\ell}=\frac{V_0(1-2^{-\alpha})}{1-2^{\alpha-r}}.$$

\begin{lem}\label{lem:1}
	Let $X$ be a random variable with zero mean, and let $\bar{X}_N$ be an average of $N$ iid samples of $X$. If $\mbe[\abs{X}^p]<\infty$ for $p>2$, then there exists a constant $C_p$ depending only on $p$ such that
	$$\mbe[\abs{\bar{X}_N}^p]\le C_p \frac{\mbe[\abs{X}^p]}{N^{p/2}}.$$
\end{lem}

\Cref{lem:1} is stated as Lemma~1 in \cite{GG2019}, with which we have the following theorem.

\begin{theorem}\label{thm:sfmc}
	Suppose that $f(x;\yobs)>0$ w.p.1, and there exist $p,q>2$ with $(p-2)(q-2)>4$ such that
	$$\mbe\left[\abs{\frac{f(x;\yobs)}{p(\yobs|\theta)}}^p\right]<\infty \text{ and }\mbe\left[\left(1+\abs{\log \frac{f(x;\yobs)}{p(\yobs|\theta)}}^q\right)||\nabla_\lambda \log q_\lambda(\theta)||_2^q\right]<\infty,$$
	where the expectations are taken with respect to $(\theta,x)\sim q_\lambda(\theta)p(x|\theta)$, then $$\mbe[\Delta \psi_{\theta,\ell}^2||\nabla_\lambda \log q_\lambda(\theta)||_2^2] = O(2^{-r\ell}) \text{ with } r=\min\left(\frac{p(q-2)}{2q},2\right)\in(1,2].$$
\end{theorem}
\begin{proof}
	This proof is in line with Theorem 2 of \cite{GHI:2020}, which developed MLMC for a nested expectation of the form $\mbe_{X,Y}[\log[g(X,Y)|Y]]$. Let
	$$R = \frac 1 {M_{\ell}}\sum_{i=1}^{M_{\ell}} \frac{f(x_i;\yobs)}{p(\yobs|\theta)},$$
	$$R^{(a)} = \frac 1 {M_{\ell-1}}\sum_{i=1}^{M_{\ell-1}} \frac{f(x_i;\yobs)}{p(\yobs|\theta)},\ R^{(b)} =\frac 1 {M_{\ell-1}}\sum_{i=M_{\ell-1}+1}^{M_\ell} \frac{f(x_i;\yobs)}{p(\yobs|\theta)}.$$
	We then have
	$$\Delta \psi_{\theta,\ell} = (\log R-R+1)-\frac{1}{2}\left[(\log R^{(a)}-R^{(a)}+1)+(\log R^{(b)}-R^{(b)}+1)\right].$$
	Applying Jensen's inequality gives
	$$\Delta \psi_{\theta,\ell}^2\le 2(\log R-R+1)^2+(\log R^{(a)}-R^{(a)}+1)^2+(\log R^{(b)}-R^{(b)}+1)^2.$$

	Note that $|\log x-x+1|\le |x-1|^r\max(-\log x,1)$ for any $x>0$ and any $1< r\le 2$.  By Holder's inequality, we have
	\begin{align*}
	\mbe[(\log R-R+1)^2&||\nabla_\lambda \log q_\lambda(\theta)||_2^2]\le \mbe[(R-1)^{2r}\max(-\log R,1)^2||\nabla_\lambda \log q_\lambda(\theta)||_2^2]\\
	&\le \mbe[(R-1)^{2rs}]^{1/s}\mbe[\max(-\log R,1)^{2t}||\nabla_\lambda \log q_\lambda(\theta)||_2^{2t}]^{1/t}
	\end{align*}
	for any $s,t\ge 1$ satisfying $1/s+1/t=1$.
	
	Note that $\mbe[R-1]=0$. Hence, if $2rs\le p$, then it follows from \Cref{lem:1} that
	$$\mbe[(R-1)^{2rs}]\le \frac{C_{2sr}}{M_\ell^{sr}}\mbe\left[|f(x;\yobs)/p(\yobs|\theta)-1|^{2rs}\right],$$
	where $\mbe\left[|f(x;\yobs)/p(\yobs|\theta)-1|^{2rs}\right]<\infty$.
	Notice that the function $\max(-\log x,1)^{2t}$ is convex for $x>0$. Thus, applying Jensen's inequality and using $f(x_i;\yobs)>0$, we have
	\begin{align*}
	\max(-\log R,1)^{2t}&=\max\left(-\log \frac 1 {M_{\ell}}\sum_{i=1}^{M_{\ell}} \frac{f(x_i;\yobs)}{p(\yobs|\theta)},1\right)^{2t}\\
	&\le\frac 1 {M_{\ell}}\sum_{i=1}^{M_{\ell}} \max\left(-\log  \frac{f(x_i;\yobs)}{p(\yobs|\theta)},1\right)^{2t}\\
	&\le 1+\frac 1 {M_{\ell}}\sum_{i=1}^{M_{\ell}}\abs{\log  \frac{f(x_i;\yobs)}{p(\yobs|\theta)}}^{2t}.
	\end{align*}

	As a result, as long as $2t\le q$, we have
	\begin{equation}\label{eq:Jensen}
    \begin{aligned}
	&\mbe[\max(-\log R,1)^{2t}||\nabla_\lambda \log q_\lambda(\theta)||_2^{2t}]\\
    &\le \mbe\left[\left(1+\abs{\log  \frac{f(x;\yobs)}{p(\yobs|\theta)}}^{2t}\right)||\nabla_\lambda \log q_\lambda(\theta)||_2^{2t}\right]<\infty.
    \end{aligned}
	\end{equation}

	Particularly, we take $s=q/(q-2)$, $t=q/2$ and $r=\min(p(q-2)/(2q),2)$. Since $(p-2)(q-2)> 4$, $r> 1$. Therefore, $\mbe[(\log R-R+1)^2||\nabla_\lambda \log q_\lambda(\theta)||_2^2]=O(M_\ell^{-r})$. This argument holds also by replacing $R$ with $R^{(a)}$ or $R^{(b)}$. We thus have $\mbe[\Delta \psi_{\theta,\ell}^2||\nabla_\lambda \log q_\lambda(\theta)||_2^2] = O(M_\ell^{-r}) = O(2^{-r\ell})$.
\end{proof}

It should be noticed that \Cref{thm:sfmc} requires $f(x;\yobs)>0$ w.p.1.  If not, the inequalities in \cref{eq:Jensen} do not hold. This implies that our result rules out the case of  indicator functions in formulating likelihoods.

\subsection{Re-parameterization gradient}
Assume that there exists a transformation $x= \Lambda(\bm v;\theta)\sim p(x|\theta)$, where $\bm v\sim p_2(\bm v)$ independently of $\theta$ and $\nabla_{\theta} \Lambda(\bm v;\theta)$ exists. Using $\theta=\Gamma(\bm u;\lambda)$ as before gives  $x=\Lambda(\bm v;\Gamma(\bm u;\lambda))$.
Allowing the interchange of expectation and differentiation, the gradient \cref{eq:rep} is then rewritten as
\begin{align*}
&\nabla_\lambda L(\lambda) =\mathbb{E}_{\bm u}[\nabla_\lambda \Gamma(\bm u;\lambda)\cdot (\nabla_\theta\log \mbe_{x}[f(x;\yobs)]+ \nabla_\theta\log p(\theta)-\nabla_\theta\log q_\lambda(\theta))]\notag\\
&=\mathbb{E}_{\bm u}[\nabla_\lambda \Gamma(\bm u;\lambda)\cdot (\nabla_\theta\log \mbe_{\bm v}[f(x;\yobs)]+ \nabla_\theta\log p(\theta)-\nabla_\theta\log q_\lambda(\theta))]\notag\\
&=\mathbb{E}_{\bm u}\left[\nabla_\lambda \Gamma(\bm u;\lambda)\cdot \left(\frac{\mbe_{\bm v}[\nabla_\theta f(x;\yobs)]}{\mbe_{\bm v}[f(x;\yobs)]}+ \nabla_\theta\log p(\theta)-\nabla_\theta\log q_\lambda(\theta)\right)\right]\notag\\
&=\mathbb{E}_{\bm u}\left[\nabla_\lambda \Gamma(\bm u;\lambda)\cdot \left(\frac{\mbe_{\bm v}[\nabla_\theta \Lambda(\bm v;\theta)\nabla_xf(x;\yobs)]}{\mbe_{\bm v}[f(x;\yobs)]}+ \nabla_\theta\log p(\theta)-\nabla_\theta\log q_\lambda(\theta)\right)\right],\label{eq:RP}
\end{align*}
where $\nabla_\lambda\Gamma(\bm u;\lambda)$ is the Jacobian matrix with entries $[\nabla_\lambda\Gamma(\bm u;\lambda)]_{ij}=\partial\Gamma_j(\bm u;\lambda)/\partial\lambda_i$.

Define
\begin{equation}
\mathrm{RP}_{N}(\lambda) =\nabla_\lambda \Gamma(\bm u;\lambda)\cdot \left(\frac{\nabla_\theta \hat{p}_N(\yobs|\theta)}{\hat{p}_N(\yobs|\theta)}+ \nabla_\theta\log p(\theta)-\nabla_\theta\log q_\lambda(\theta)\right),\label{eq:nestedRP}
\end{equation}
where
$$\hat{p}_N(\yobs|\theta) = \frac{1}{N}\sum_{i=1}^Nf(x_i;\yobs)\text{ with }x_i=\Lambda(\bm v_i;\theta),$$
$$\nabla_\theta\hat{p}_N(\yobs|\theta) = \frac{1}{N}\sum_{i=1}^N\nabla_\theta f(x_i;\yobs)=\frac{1}{N}\sum_{i=1}^N\nabla_\theta \Lambda(\bm v_i;\theta)\nabla_x f(x_i;\yobs),$$
with $[\nabla_\theta \Lambda(\bm v;\theta)]_{ij}=\partial \Lambda_j(\bm v;\theta)/\partial \theta_i$ and $\bm v_i\sim p_2(\bm v)$ independently.
The estimator \eqref{eq:nestedRP} is also biased. Now we take
$$\tilde{\psi}_{\theta,M_\ell}=\frac{\nabla_\theta\hat{p}_{M_\ell}(\yobs|\theta)}{\hat{p}_{M_\ell}(\yobs|\theta)}=\frac{\sum_{i=1}^{M_\ell}\nabla_\theta \Lambda(\bm v_i;\theta)\nabla_x f(x_i;\yobs)}{\sum_{i=1}^{M_\ell}f(x_i;\yobs)},$$
to differ from $\psi_{\theta,M_\ell}$ in the SF method.
Analogously, we take $\Delta \tilde \psi_{\theta,0}=\tilde\psi_{\theta,M_0}$. For $\ell\ge 1$, we use an antithetic coupling estimator again
\begin{equation}\label{eq:RPdelta}
\Delta \tilde\psi_{\theta,\ell}=\tilde\psi_{\theta,M_\ell}-\frac{1}{2}\left(\tilde\psi_{\theta,M_{\ell-1}}^{(a)}+\tilde\psi_{\theta,M_{\ell-1}}^{(b)}\right),
\end{equation}
where
\begin{align*}
&\tilde\psi_{\theta,M_{\ell-1}}^{(a)} = \frac{\sum_{i=1}^{M_{\ell-1}}\nabla_\theta \Lambda(\bm v_i;\theta)\nabla_x f(x_i;\yobs)}{\sum_{i=1}^{M_{\ell-1}}f(x_i;\yobs)},\\ &\tilde\psi_{\theta,M_{\ell-1}}^{(b)} = \frac{\sum_{i=M_{\ell-1}+1}^{M_{\ell}}\nabla_\theta \Lambda(\bm v_i;\theta)\nabla_x f(x_i;\yobs)}{\sum_{i=M_{\ell-1}+1}^{M_{\ell}}f(x_i;\yobs)}.
\end{align*}
Define
\begin{equation}\label{eq:rp}
\mathrm{RP}_{\text{MLMC}}(\lambda) = \nabla_\lambda \Gamma(\bm u;\lambda)\cdot \left(\frac{\Delta \tilde\psi_{\theta,I}}{w_I}+ \nabla_\theta\log p(\theta)-\nabla_\theta\log q_\lambda(\theta)\right),
\end{equation}
where $\theta = \Gamma(\bm u;\lambda)$ and $w_I$ is defined as in the SF method. For any number of outer samples $S\ge 1$, the gradient estimator
$$\widehat{\nabla_\lambda L}^{\mathrm{RP}}(\lambda) =\frac{1}{S} \sum_{i=1}^S \mathrm{RP}_{\text{MLMC}}^{(i)}(\lambda),$$
is unbiased, where $\mathrm{RP}_{\text{MLMC}}^{(i)}(\lambda)$ are iid copy of $\mathrm{RP}_{\text{MLMC}}(\lambda)$.

Similarly, to ensure a finite variance and finite expected computational cost of $\mathrm{RP}_{\text{MLMC}}(\lambda)$, it suffices to show $\mbe[||\nabla_\lambda \Gamma(\bm u;\lambda)\cdot\Delta \tilde \psi_{\theta,\ell}||_2^2] = O(2^{-r\ell})$ for some $r>1$. This can be achieved as shown in the following theorem.

\begin{theorem}\label{thm:remc}
	If
	$$\sup_{x} ||\nabla_\lambda \log f(x;\yobs))||_\infty<\infty,$$
	where $x=\Lambda(\bm v;\Gamma(\bm u;\lambda))$, and assume that there exists $p>2$ such that
	$$\mbe\left[\abs{\frac{f(x,\yobs)}{p(\yobs|\theta)}}^p\right]<\infty,$$
	then $$\mbe[||\nabla_\lambda \Gamma(\bm u;\lambda)\Delta \tilde \psi_{\theta,\ell}||_2^2] = O(2^{-r\ell}) \text{ with } r=\min(p/2,2)\in(1,2].$$
\end{theorem}

\begin{proof}
	The proof follows an argument similar to Theorem 3.1 in \cite{goda:2020}, which considered a nested expectation involving a ratio of two inner conditional expectations.
\end{proof}

\section{Parameterizations in Gaussian variational family}\label{Gaussian}
Throughout this paper, we use the Gaussian family $N(\mu,\Sigma)$ as the variational family. For the SF method, we take the variational parameters as  $\lambda = (\mu,\vech(C))$, where $C$ is the Cholesky decomposition (lower triangular) of $\Sigma^{-1}$ and $\vech(C)$ denotes a
vector obtained by stacking the
lower triangular elements of $C$. The number of variational parameters $d_\lambda = p+p(p+1)/2$.  Since $\log q_\lambda(\theta) = \log |\det(C)|-\frac 1 2(\theta-\mu)^\top CC^\top(\theta-\mu)$, $\nabla_\lambda \log q_\lambda(\theta)=(\nabla_\mu  \log q_\lambda(\theta),\nabla_{\vech(C)} \log q_\lambda(\theta))$ with
\begin{align*}
\nabla_\mu  \log q_\lambda(\theta) &=  CC^\top(\theta-\mu),\\
\nabla_{\vech(C)} \log q_\lambda(\theta) &= \vech(\diag{1/C}-(\theta-\mu)(\theta-\mu)^\top C),
\end{align*}
where $\diag{1/C}$ denotes the diagonal matrix with the same dimensions as $C$ with $i$th diagonal entry $1/C_{ii}$. Note that the score function $\nabla_\lambda \log q_\lambda(\theta)$ is model-free. The SF estimator $\mathrm{SF}_{N}(\lambda)$ can be easily obtained by \cref{eq:nested}. It is common to use control variate (CV) to reduce the noise in estimating the  gradient \cite{MF:2017,PBJ:2012}.
Note that $\mbe[\nabla_\lambda\log q_\lambda(\theta)]=0$. For any constant vector $c=(c_1,\dots,c_p)\in\mbr^p$, the estimator is also unbiased for the gradient,
\begin{equation*}
\mathrm{SF}^{\text{CV}}_{\text{MLMC}}(\lambda,c) = \nabla_\lambda \log q_\lambda(\theta) \left[\frac{\Delta \psi_{\theta,I}}{w_I}+ \log p(\theta)-\log q_\lambda(\theta)-c\right].
\end{equation*}
We can take an optimal $c_i$ to minimize the variance of the $i$th entry of $\mathrm{SF}^{\text{CV}}_{\text{MLMC}}(\lambda,c)$. Solving
$$c^*_i = \arg \min_{c_i\in \mbr} \var{\mathrm{SF}^{\text{CV}}_{\text{MLMC},i}(\lambda,c_i)}$$
gives
\begin{equation}
c^*_i = \frac{\mbe[\left(\nabla_{\lambda_i} \log q_\lambda(\theta)\right)^2\xi]}{\mbe[\left(\nabla_{\lambda_i} \log q_\lambda(\theta)\right)^2]}=\frac{\cov{\nabla_{\lambda_i} \log q_\lambda(\theta),\nabla_{\lambda_i} \log q_\lambda(\theta)\xi}}{\var{\nabla_{\lambda_i} \log q_\lambda(\theta)}},\label{eq:cv}
\end{equation}
where $\xi=\frac{\Delta \psi_{\theta,I}}{w_I}+ \log p(\theta)-\log q_\lambda(\theta)$.
In practice, $c^*_i$ ($i=1,\dots,p$) are estimated by using the samples in the previous iteration.
The whole procedure is summarized in \Cref{alg:em1}.

\begin{algorithm}
	\caption{Unbiased MLMC with the SF gradient estimator\label{alg:em1}}
	\begin{algorithmic}[1]
		\STATE Initialize $\lambda^{(0)} = (\mu^{(0)},\vech(C^{(0)}))$, $t=0$, $M$ the number of outer samples, $\alpha\in(1,r)$ and $w_\ell\propto 2^{-\alpha \ell}$ such that $\sum_{\ell=0}^\infty w_\ell =1$ and all $w_\ell>0$.
		
		\STATE Repeat
		
	    (a) Generate $\theta^{(t)}_1,\dots,\theta^{(t)}_m\sim N(\mu^{(t)},(C^{(t)}{C^{(t)}}^\top)^{-1})$ independently and $I_1^{(t)},\dots,I_m^{(t)}$ independently and randomly with probability $w_\ell$.
		
		(b) Let $n_i = M_02^{I_i^{(t)}}$. For $i=1,\dots,m$, generate $x_{i1}^{(t)},\dots,x_{in_i}^{(t)}\sim p(x|\theta^{(t)}_i)$ independently. Compute the associated samples of the correction $\Delta \psi_{\theta,I}$, denoted by $\Delta \psi^{(t)}_i$, $i=1,\dots,m$.
		
		(c) Estimate $c^*$ defined by \eqref{eq:cv} by the samples $\theta^{(t)}_i,\ I_i^{(t)},\ \Delta \psi^{(t)}_i$, $i=1,\dots,m$, resulting in $c^{(t)}$.
		
		(d) If $t>0$, compute the gradient estimator
        \begin{align*} 		
		&\widehat{\nabla_\lambda L}^{\mathrm{SF}}(\lambda^{(t)}) \\
        = &\frac{1}{m}\sum_{i=1}^m \nabla_\lambda \log q_\lambda(\theta^{(t)}_i) \left[\frac{	\Delta \psi^{(t)}_i}{w_{I_i^{(t)}}}+ \log p(\theta^{(t)}_i)-\log q_{\lambda^{(t)}}(\theta^{(t)}_i)-c^{(t-1)}\right],
        \end{align*}
		and the ELBO estimator
		$$\mathrm{LB}(\lambda^{(t)}) = \frac{1}{S}\sum_{i=1}^S \frac{	\Delta \psi^{(t)}_i}{w_{I_i^{(t)}}}+ \log p(\theta^{(t)}_i)-\log q_{\lambda^{(t)}}(\theta^{(t)}_i).$$
		Update the VB parameter:
		$$\lambda^{(t+1)} = \lambda^{(t)}+\rho_t \widehat{\nabla_\lambda L}^{\mathrm{SF}}(\lambda^{(t)}).$$
		If $t=0$, then set $\lambda^{(t+1)} = \lambda^{(t)}$. Note that this step is used to initialize $c^*$ rather than updating the VB parameter.
		
		(e) $t = t+1$
		
		until some stopping rule is satisfied.
\end{algorithmic}
\end{algorithm}

Using the RP method, we take the variational parameter as  $\lambda = (\mu,\vech(L))$, where $L$ is the Cholesky decomposition of $\Sigma$, which is different from the parameterizations in the SF method. For this case, $\theta =  \Gamma(\bm u;\lambda) = \mu+L\bm u \sim N(\mu,\Sigma)$, where $\bm u\in \mathbb{R}^{p\times 1}$ is a standard normal. Let
$$G = \frac{\Delta \tilde\psi_{\theta,I}}{w_I}+ \nabla_\theta\log p(\theta)-\nabla_\theta\log q_\lambda(\theta)\in \mathbb{R}^{p\times 1},$$
where $\Delta \tilde\psi_{\theta,\ell}$ is given by \cref{eq:RPdelta} and $\nabla_\theta\log q_\lambda(\theta)=-\Sigma^{-1}(\theta-\mu)=-(LL^\top)^{-1}(\theta-\mu)$. Then the RP estimator is given by $$\mathrm{RP}_{\text{MLMC}}(\lambda)=(G,\vech(G\bm u^\top))\in \mathbb{R}^{d_\lambda\times 1}.$$
The second term $\nabla_\theta\log p(\theta)$ in $G$ depends on the prior. Particularly, if the prior is normally distributed, say, $N(\mu_0,\Sigma_0)$, then $\nabla_\theta\log p(\theta)=-\Sigma_0^{-1}(\theta-\mu_0)$. It is crucial to work out the term $\nabla_\theta \Lambda(\bm v;\theta)\nabla_x f(x;\yobs)$ used in $\Delta \tilde\psi_{\theta,\ell}$, which is model-specific. The whole procedure for the RP method is summarized in \Cref{alg:em2}.

\begin{algorithm}
	\caption{Unbiased MLMC with the RP estimator\label{alg:em2}}
	\begin{algorithmic}[1]
		\STATE Initialize $\lambda^{(0)} = (\mu^{(0)},\vech(L^{(0)}))$, $t=0$, $M$ the number of outer samples, $\alpha\in(1,r)$ and $w_\ell\propto 2^{-\alpha \ell}$ such that $\sum_{\ell=0}^\infty w_\ell =1$ and all $w_\ell>0$.
		
		\STATE Repeat
		
		(a) Generate $\bm{u}_1^{(t)},\dots, \bm{u}_m^{(t)}\sim N(0,I_p)$ independently and set $\theta^{(t)}_i = \mu^{(t)}+L^{(t)}\bm u_i^{(t)}$. Generate $I_1^{(t)},\dots,I_m^{(t)}$ independently and randomly with probability $w_\ell$.
		
		(b) Let $n_i = M_02^{I_i^{(t)}}$. For $i=1,\dots,m$, generate
		$\bm{v}_{i1}^{(t)},\dots,\bm{v}_{in_i}^{(t)}\sim p_2(\bm v)$ independently
		and set $x_{ij}^{(t)} = \Lambda(\bm{v}_{ij}^{(t)};\theta^{(t)}_i)$, $j=1,\dots,n_i$. Compute the associated samples of the corrections $\Delta \psi_{\theta,I}$ and $\Delta \tilde\psi_{\theta,I}$, denoted by $\Delta \psi^{(t)}_i$ and $\Delta \tilde\psi^{(t)}_i$, respectively.
		
		(c) Compute the gradient estimator 		
		$$\widehat{\nabla_\lambda L}^{\mathrm{RP}}(\lambda^{(t)}) = \frac{1}{m}\sum_{i=1}^m (G_i^{(t)},\vech(G_i^{(t)}\bm u_i^\top)),$$
		where $$G_i^{(t)}= \frac{\Delta \tilde\psi^{(t)}_i}{w_{I_i^{(t)}}}+ \nabla_\theta\log p(\theta^{(t)}_i)-\nabla_\theta\log q_{\lambda^{(t)}}(\theta^{(t)}_i),$$
		and compute the ELBO estimator
		$$\mathrm{LB}(\lambda^{(t)}) = \frac{1}{S}\sum_{i=1}^S \frac{	\Delta \psi^{(t)}_i}{w_{I_i^{(t)}}}+ \log p(\theta^{(t)}_i)-\log q_{\lambda^{(t)}}(\theta^{(t)}_i).$$
		Update the VB parameter:
		$$\lambda^{(t+1)} = \lambda^{(t)}+\rho_t \widehat{\nabla_\lambda L}^{\mathrm{RP}}(\lambda^{(t)}).$$

		(d) $t = t+1$
		
		until some stopping rule is satisfied.
	\end{algorithmic}
\end{algorithm}

Notice that not only the gradient estimators but also the ELBO estimators are unbiased in Algorithms \Cref{alg:em1,alg:em2}. The unbiased MLMC methods can be expected to estimate the ELBO more accurately.

\section{Incorporating RQMC}\label{se:RQMC}

We now incorporate RQMC sampling based scrambled $(t,s)$-sequences within the MLMC estimators. Quasi-Monte Carlo (QMC) is designed for computing expectations of $f(\bm v)$ for $\bm v\sim U[0,1]^s$. We should note that in our present context, the underlying distributions are not the form of uniforms. To fit QMC in practice, one must transform the base distribution $U[0,1]^s$ to the underlying distributions. Suppose that there exists a transformation $\psi(\cdot)$ such $\psi(\bm v)\sim p$, where $p$ is the underlying distribution. Below we subsume any such transformation $\psi(\cdot)$ into the definition of $f$.

To estimate $\mu = \int_{[0,1]^s}f(\bm v)\mrd \bm v$, QMC methods use a sample-mean estimator
$$\hat \mu =\frac{1}{N}\sum_{i=1}^N f(\bm v_i),$$
where $\bm v_1,\dots,\bm v_N$ are the first $N$ points of a low discrepancy sequence. By the Koksma-Hlawka inequality, we have
$$\abs{\hat \mu-\mu}\le \hk(f)D^*(\bm v_1,\dots,\bm v_N),$$
where $\hk(f)$ is the variation of the integrand $f(\cdot)$ in the sense of Hardy and Krause, and $D^*(\bm v_1,\dots,\bm v_N)$ is the star discrepancy of the point set $\{\bm v_1,\dots,\bm v_N\}$. For $(t,s)$-sequences, we have $$D^*(\bm v_1,\dots,\bm v_N)=O(N^{-1}(\log N)^{s})=O(N^{-1+\epsilon}),$$ where we use an arbitrarily small $\epsilon>0$ for hiding the logarithm term throughout this paper. If $f$ is of bounded variation in the sense of Hardy and Krause (BVHK), one gets a QMC error of $O(N^{-1+\epsilon})$.
To get a practical error estimate, RQMC methods were introduced, see \cite{LEcuyer2005} for a review. In this paper, we use the scrambling technique proposed by \cite{Owen1995} to randomize  $(t,s)$-sequences.  In RQMC, each $\bm v_i\sim U[0,1]^s$ marginally, implying that $\hat \mu$ is unbiased for $\mu$. More importantly, scrambled $(t,s)$-sequence retains a $(t,s)$-sequence w.p.1. This leads to
$$\var{\hat \mu}=\mbe[(\hat \mu-\mu)^2]\le \hk(f)^2D^*(\bm v_1,\dots,\bm v_N)^2,$$
where the expectation is taken with respect to the randomness of scrambling. Apparently, the RQMC variance is of $O(N^{-2+\epsilon})$ if $f$ is of BVHK.

Now we focus on how to incorporate RQMC within the MLMC estimators. In fact, for both the SF and RP  estimators, one needs to sample $\theta\sim q_\lambda(\theta)$, $x_1,\dots,x_{M_I}\sim p(x|\theta,I)$ and $I$ from a discrete distribution with $P(I=i)=w_i$ as stated above. For each realization, the number of random variables depends on $I$, which takes values in $\mbn$. It is not possible to use a scrambled $(t,s)$-sequence to sample all random variables in a single run because we need determine the dimension $s$ in advance. Instead, we use hybrid sequences within the MLMC estimators. Specifically,  we still use MC to sample $\theta$ and $I$, but use RQMC in inner simulation. That is, $x_1,\dots,x_{M_I}$ is based on a scrambled $(t,s)$-sequence. To this end, we assume that there exists a transformation $\Lambda$ such that
$$x=\Lambda(\bm v;\theta)\sim p(x|\theta),$$
where $\bm v\sim U[0,1]^s$.  We then takes $x_i = \Lambda(\bm v_i;\theta)$ in the inner simulation, where $\bm v_1,\dots,\bm v_{M_I}$ are the first $M_I$ points of a scrambled $(t,s)$-sequence. Since RQMC estimates are unbiased, the replacement of RQMC will not change the unbiasedness of the gradient estimators.

We are ready to establish an RQMC version of \Cref{thm:sfmc} for the SF gradient. We should note that \Cref{thm:sfmc} may not be extended to the RQMC setting since \Cref{lem:1} holds only for iid samples.  Recently, for proving strong law of large numbers for scrambled net
integration, \cite{OR:2021} showed that $\mbe[\abs{\bar{X}_N}^p]\le C_pN^{1-p}$ for $p\in(1,2)$ via the Riesz-Thorin interpolation theorem, where $\bar{X}_N$ is an average of $N$ RQMC samples of $X$ with $\mbe[X]=0$. However, this result is not for the case $p>2$ required in \Cref{lem:1}. It is not clear whether the RQMC version of \Cref{lem:1} holds. This is left for future research. The theorem we provide below is totally different from \Cref{thm:sfmc}, and the proof of which does not depend on \Cref{lem:1}.

\begin{theorem}\label{thm:sfqmc}
	Suppose that samples $x_i = \Lambda(\bm v_i;\theta),i=1,\dots,M_\ell$ are used in the SF estimator \cref{eq:SF}, where $\bm v_i\in [0,1]^s$ are the first $M_\ell$ points of a scrambled $(t,s)$-sequence. If
	$$\mbe\left[\frac{\hk(f_\theta)^2||\nabla_\lambda \log q_\lambda(\theta)||_2^2}{f_\theta(\bm v)^2}\right]<\infty,$$
	where $\bm v\sim U[0,1]^s$, $f_\theta(\bm v) = f(\Lambda(\bm v;\theta);\yobs)$, and $\Lambda(\bm v;\theta)\sim p(x|\theta)$,
	then we have
	$$\mbe[\Delta \psi_{\theta,\ell}^2||\nabla_\lambda \log q_\lambda(\theta)||_2^2] = O(2^{-r\ell}) \text{ with } r=2-\epsilon$$
	for arbitrarily small $\epsilon>0$.
\end{theorem}

\begin{proof}
Note that $p(\yobs|\theta) =\mbe[f_\theta(\bm v)|\theta]$. Let $$P_\ell = \frac{1}{M_\ell} \sum_{i=1}^{M_{\ell}} f(x_i;\yobs)=\frac{1}{M_\ell} \sum_{i=1}^{M_{\ell}} f_\theta(\bm v_i),$$
with $P_{\ell-1}^{(a)} = \frac{1}{M_{\ell-1}}\sum_{i=1}^{M_{\ell-1}}f_\theta(\bm v_i)$, and $P_{\ell-1}^{(b)} = \frac{1}{M_{\ell-1}}\sum_{i=M_{\ell-1}+1}^{M_{\ell}} f_\theta(\bm v_i)$.
All of them are RQMC estimators for $p(\yobs|\theta)$. We have
	$$\Delta \psi_{\theta,\ell} = [\log P_\ell-\log p(\yobs|\theta)]-\frac{1}{2}[(\log P_\ell^{(a)}-\log p(\yobs|\theta))+(\log P_\ell^{(b)}-\log p(\yobs|\theta))].$$
	Applying Jensen's inequality gives
	$$\Delta \psi_{\theta,\ell}^2\le 2(\log P_\ell-\log p(\yobs|\theta))^2+(\log P_\ell^{(a)}-\log p(\yobs|\theta))^2+(\log P_\ell^{(b)}-\log p(\yobs|\theta))^2.$$
	
	Note that $\abs{\log t}\le \max(1,1/t)\abs{t-1}\le (1+1/t)\abs{t-1}$ for any $t>0$. We thus have
	$$\abs{\log P_\ell-\log p(\yobs|\theta)}\le (1/p(\yobs|\theta)+ 1/P_\ell)\abs{P_\ell-p(\yobs|\theta)}.$$
	By the  Koksma-Hlawka inequality, we have
	$$\abs{P_\ell-p(\yobs|\theta)}\le \hk(f_\theta)D_{\ell},$$
	where $D_{\ell}=D^*(\bm v_1,\dots,\bm v_{M_\ell})$.
	This implies that
	$$(\log P_\ell-\log p(\yobs|\theta))^2\le 2\hk(f_\theta)^2D_\ell^2\left(\frac 1{p(\yobs|\theta)^2}+ \frac 1{P_\ell^2}\right).$$
	Let $H(\theta)=\hk(f_\theta)||\nabla_\lambda \log q_\lambda(\theta)||_2$. We then have
	\begin{align*}
	\mbe[(\log P_\ell-\log p(\yobs|\theta))^2&||\nabla_\lambda \log q_\lambda(\theta)||_2^2] \le 2D_{\ell}^2\left(\mbe\left[\frac {H(\theta)^2}{p(\yobs|\theta)^2}\right]+ \mbe\left[\frac {H(\theta)^2}{P_\ell^2}\right]\right).
	\end{align*}
	
   By Jensen's inequality, we have
\begin{equation}\label{eq:pl}
\frac{1}{P_\ell^2}=\left(\frac{1}{\frac{1}{M_\ell} \sum_{i=1}^{M_{\ell}} f_\theta(\bm v_i) }\right)^2\le \frac{1}{M_\ell} \sum_{i=1}^{M_{\ell}} \frac{1}{ f_\theta(\bm v_i)^2}.
\end{equation}
By the unbiasedness of RQMC estimators and the law of total expectation,
	\begin{align*}
\mbe\left[\frac {H(\theta)^2}{P_\ell^2}\right]&\le \mbe\left[\frac{H(\theta)^2}{M_\ell} \sum_{i=1}^{M_{\ell}} \frac{1}{f(x_i;\yobs)^2}\right]\\&=\mbe\left[H(\theta)^2\mbe\left[\frac{1}{M_\ell} \sum_{i=1}^{M_{\ell}} \frac{1}{f(x_i;\yobs)^2}\bigg|\theta\right]\right]\\&=\mbe\left[H(\theta)^2\mbe\left[\frac {1}{f(x;\yobs)^2}\bigg |\theta\right]\right]=\mbe\left[\frac {H(\theta)^2}{f(x;\yobs)^2}\right]<\infty.
	\end{align*}
On the other hand, by using Jensen's inequality  and the law of total expectation again,
\begin{align*}
	&\mbe\left[\frac {H(\theta)^2}{p(\yobs|\theta)^2}\right]=\mbe\left[\frac {H(\theta)^2}{\mbe[f(x;\yobs)|\theta]^2}\right]\\
\le &\mbe\left[ H(\theta)^2\mbe\left[\frac{1}{f(x;\yobs)^2}\bigg|\theta\right]\right]=\mbe\left[ \frac{H(\theta)^2}{f(x;\yobs)^2}\right]<\infty.
\end{align*}	
	We therefore have
$$\mbe[(\log P_\ell-\log p(\yobs|\theta))^2||\nabla_\lambda \log q_\lambda(\theta)||_2^2]=O(D_\ell^2) = O(M_\ell^{-2+\epsilon})=O(2^{-r\ell})$$
for $r=2-\epsilon$ and any $\epsilon>0$. This argument holds also by replacing $P_\ell$ with $P_\ell^{(a)}$ or $P_\ell^{(b)}$. We thus have $\mbe[\Delta \psi_{\theta,\ell}^2||\nabla_\lambda \log q_\lambda(\theta)||_2^2] = O(2^{-r\ell})$.
\end{proof}

We next establish an RQMC version of \Cref{thm:remc} for the RP gradient. \Cref{thm:remc} cannot be extended to the RQMC setting since its proof depends on \Cref{lem:1} as well.

\begin{theorem}\label{thm:rpqmc}
Suppose that samples $x_i = \Lambda(\bm v_i;\theta),\ i=1,\dots,M_\ell$ in the RP estimator \cref{eq:rp}, where $\bm v_i\in [0,1]^s$ are the first $M_\ell$ points of a scrambled $(t,s)$-sequence. If
$$\mbe\left[\frac{||\nabla_\lambda \Gamma(\bm u;\lambda)||_{\max}^2}{p(\yobs|\theta)^2}\left(\frac{(\norm{\nabla p(\yobs|\theta)}_2^2+\norm{\hk(\nabla_{\theta} f_\theta)}_2^2)\hk(f_\theta)^2}{f_\theta(\bm v)^2}+\norm{\hk(\nabla_{\theta} f_\theta)}_2^2\right) \right]$$
is finite, where $\bm v\sim U[0,1]^s$, $f_\theta(\bm v) = f(\Lambda(\bm v;\theta);\yobs)$, $\Lambda(\bm v;\theta)\sim p(x|\theta)$, $\theta =\Gamma(\bm u;\lambda)\sim q_\lambda(\theta)$, $\hk(\nabla_{\theta} f_\theta)$ denotes a vector of $\hk(\partial_{\theta_i} f_\theta)$, and $||A||_{\max}$ denotes the largest absolute value of the entries of the matrix $A$, we have
$$\mbe[||\nabla_\lambda \Gamma(\bm u;\lambda)\Delta \tilde \psi_{\theta,\ell}||_2^2] = O(2^{-r\ell}) \text{ with } r=2-\epsilon$$
for arbitrarily small $\epsilon>0$.
\end{theorem}

\begin{proof}
	We use the notations $P_\ell$, $P_{\ell-1}^{(a)}$ and $P_{\ell-1}^{(b)}$ defined in the proof of \Cref{thm:sfqmc}, and define
	$$\mathcal{N}_\ell=\frac{1}{M_\ell}\sum_{i=1}^{M_\ell}\nabla_\theta \Lambda(\bm v_i;\theta)\nabla_x f(x_i;\yobs)=\frac{1}{M_\ell}\sum_{i=1}^{M_\ell}\nabla_{\theta} f_\theta(\bm v_i),$$
with $\mathcal{N}_{\ell-1}^{(a)} = \frac{1}{M_{\ell-1}}\sum_{i=1}^{M_{\ell-1}}\nabla_{\theta} f_\theta(\bm v_i)$, and $\mathcal{N}_{\ell-1}^{(b)} = \frac{1}{M_{\ell-1}}\sum_{i=M_{\ell-1}+1}^{M_{\ell}}\nabla_{\theta} f_\theta(\bm v_i)$.
It is clear that $\mbe[\mathcal{N}_\ell|\theta]=\mbe[\mathcal{N}_{\ell-1}^{(a)}|\theta]=\mbe[\mathcal{N}_{\ell-1}^{(b)}|\theta]=\nabla_\theta p(\yobs|\theta)$, and $\mbe[P_\ell|\theta]=\mbe[P_{\ell-1}^{(a)}|\theta]=\mbe[P_{\ell-1}^{(b)}|\theta]= p(\yobs|\theta)$. Note that
	\begin{align*}
		\Delta \tilde \psi_{\theta,\ell} &= \frac{\mathcal{N}_\ell}{P_\ell}-\frac{1}{2}\left(\frac{\mathcal{N}_{\ell-1}^{(a)}}{P_{\ell-1}^{(a)}}+\frac{\mathcal{N}_{\ell-1}^{(b)}}{P_{\ell-1}^{(b)}}\right)\\
		&=\left[\frac{\mathcal{N}_\ell}{P_\ell}-\frac{\nabla_\theta p(\yobs|\theta)}{p(\yobs|\theta)}\right] -\frac{1}{2}\left[\frac{\mathcal{N}_{\ell-1}^{(a)}}{P_{\ell-1}^{(a)}}-\frac{\nabla_\theta p(\yobs|\theta)}{p(\yobs|\theta)}\right]-\frac{1}{2}\left[\frac{\mathcal{N}_{\ell-1}^{(b)}}{P_{\ell-1}^{(b)}}-\frac{\nabla_\theta p(\yobs|\theta)}{p(\yobs|\theta)}\right].
	\end{align*}
	Let $\mathcal{N}_{\ell,i}$ be the $i$th entry of $\mathcal{N}_\ell$, which is an unbiased estimator for $\partial_{\theta_i} p(\yobs|\theta)$. By the triangle inequality, we find that
	\begin{align}
	\left(\frac{\mathcal{N}_{\ell,i}}{P_\ell}-\frac{\partial_{\theta_i} p(\yobs|\theta)}{p(\yobs|\theta)}\right)^2&=\left(\frac{\mathcal{N}_{\ell,i}}{P_\ell}-\frac{\mathcal{N}_{\ell,i}}{p(\yobs|\theta)}+\frac{\mathcal{N}_{\ell,i}}{p(\yobs|\theta)}-\frac{\partial_{\theta_i} p(\yobs|\theta)}{p(\yobs|\theta)}\right)^2\notag\\
	&\le \frac{2}{p(\yobs|\theta)^2}\left[\frac{\mathcal{N}_{\ell,i}^2}{P_\ell^2}(P_\ell-p(\yobs|\theta))^2+(\mathcal{N}_{\ell,i}-\partial_{\theta_i} p(\yobs|\theta))^2	\right].\label{eq:delta}
	\end{align}
	By the  Koksma-Hlawka inequality, we have
	$$\abs{P_\ell-p(\yobs|\theta)}\le \hk(f_\theta)D_{\ell},$$
	$$\abs{\mathcal{N}_{\ell,i}-\partial_{\theta_i} p(\yobs|\theta)}\le \hk(\partial_{\theta_i} f_\theta)D_{\ell},$$
	where $D_{\ell}=D^*(\bm v_1,\dots,\bm v_{M_\ell})$.

For large enough $\ell$, it is reasonable to assume that $D_\ell<1$.	
Together with \cref{eq:pl} and \cref{eq:delta}, we then have
\begin{align*}
&\left(\frac{\mathcal{N}_{\ell,i}}{P_\ell}-\frac{\partial_{\theta_i} p(\yobs|\theta)}{p(\yobs|\theta)}\right)^2\\
\le& \frac{2D_{\ell}^2}{p(\yobs|\theta)^2}\left[\frac{\mathcal{N}_{\ell,i}^2}{P_\ell^2}\hk(f_\theta)^2+\hk(\partial_{\theta_i} f_\theta)^2	\right]\\
\le& \frac{4D_{\ell}^2}{p(\yobs|\theta)^2}\left[\frac{\partial_{\theta_i} p(\yobs|\theta)^2+\hk(\partial_{\theta_i} f_\theta)^2}{P_\ell^2}\hk(f_\theta)^2+\hk(\partial_{\theta_i} f_\theta)^2	\right]\\
\le& \frac{4D_{\ell}^2}{p(\yobs|\theta)^2}\left[\frac{(\partial_{\theta_i} p(\yobs|\theta)^2+\hk(\partial_{\theta_i} f_\theta)^2)\hk(f_\theta)^2}{M_\ell} \sum_{i=1}^{M_{\ell}} \frac{1}{f_\theta(\bm v_i)^2}+\hk(\partial_{\theta_i} f_\theta)^2	\right].
\end{align*}
Let $n_r$ and $n_c$ be the number of rows and columns of the Jacobian matrix $\nabla_\lambda \Gamma(\bm u;\lambda)$, respectively, and $M_\lambda=||\nabla_\lambda \Gamma(\bm u;\lambda)||_{\max}$.  As a result,
\begin{align*}
&\mbe\left[\norm{\nabla_\lambda \Gamma(\bm u;\lambda)\cdot\left(\frac{\mathcal{N}_\ell}{P_\ell}-\frac{\nabla_\theta p(\yobs|\theta)}{p(\yobs|\theta)}\right)}_2^2\right]\\
\le& n_cn_r\mbe\left[\sum_{i=1}^{n_r}M_\lambda^2\left(\frac{\mathcal{N}_{\ell,i}}{P_\ell}-\frac{\partial_{\theta_i} p(\yobs|\theta)}{p(\yobs|\theta)}\right)^2\right]\\
\le& C_{\ell}\mbe\left[\frac{M_\lambda^2}{p(\yobs|\theta)^2}\sum_{i=1}^{n_r}\left(\frac{(\partial_{\theta_i} p(\yobs|\theta)^2+\hk(\partial_{\theta_i} f_\theta)^2)\hk(f_\theta)^2}{f_\theta(\bm v)^2}+\hk(\partial_{\theta_i} f_\theta)^2\right) \right]\\
=&C_{\ell}\mbe\left[\frac{M_\lambda^2}{p(\yobs|\theta)^2}\left(\frac{(\norm{\nabla p(\yobs|\theta)}_2^2+\norm{\hk(\nabla_{\theta} f_\theta)}_2^2)\hk(f_\theta)^2}{f_\theta(\bm v)^2}+\norm{\hk(\nabla_{\theta} f_\theta)}_2^2\right) \right]\\
=&O(M_\ell^{-2+\epsilon})=O(2^{-r\ell})
\end{align*}
with $C_\ell=4n_cn_rD_{\ell}^2$ for $r=2-\epsilon$ and any $\epsilon>0$. By a similar argument in the proof of \Cref{thm:sfqmc}, we have $\mbe[||\nabla_\lambda \Gamma(\bm u;\lambda)\Delta \tilde \psi_{\theta,\ell}||_2^2] = O(2^{-r\ell})$.
\end{proof}

In \Cref{thm:sfqmc,thm:rpqmc}, the integrands in RQMC quadratures need to be BVHK. For practical problems, it may be very hard to verify such a condition. Particularly, if the integrands are not smooth enough, the BVHK condition does not hold. For such cases, one may get a lower rate $r$. For any integrand in $L^2[0,1]^s$, scrambled nets have variance $o(1/N)$ without requiring the BVHK condition \cite{owen1997a}. Additionally, for any fixed $N$, the scrambled nets variance is no worse than a constant times the MC variance. From this point of view, under the same conditions in \Cref{thm:sfmc,thm:remc}, we can expect that the rate $r$ for RQMC is no worse than that of MC. Finally, we should note that the rates established in \Cref{thm:sfqmc,thm:rpqmc} do not benefit from the antithetic coupling, implying that the results also hold for the usual way of coupling. One might get a better rate by taking account for the form of antithetic coupling.

There are some other ways to incorporate RQMC in MLMC. For example, one can use RQMC in the outer simulation. That is, the samples of $\theta$ are based on a scrambled $(t,s')$-sequence while the inner samples $x_i$ and the samples of $I$ are based on MC. To this end, assuming $\theta = \Gamma_\lambda(\bm u)\sim q_\lambda(\theta)$ with $\bm u\sim U[0,1]^{s'}$, we take
$$\theta_i = \Gamma_\lambda(\bm u_i),\ i=1,\dots,S,$$
where $\bm u_1,\dots,\bm u_S$ are the first $S$ points of a scrambled $(t,s')$ sequence.
Taking the SF gradient estimator \cref{eq:SF_esti} for instance, we have
\begin{align}
\var{\widehat{\nabla_\lambda L}^{\mathrm{SF}}(\lambda)} &=\mbe\left[\var{\widehat{\nabla_\lambda L}^{\mathrm{SF}}(\lambda)|\theta_{\{1:S\}}}\right] + \var{\mbe[\widehat{\nabla_\lambda L}^{\mathrm{SF}}(\lambda)|\theta_{\{1:S\}}]}\notag\\
&=\frac{1}{S} \mbe[\var{\mathrm{SF}_{\text{MLMC}}(\lambda)|\theta}] +\var{\frac{1}{S}\sum_{i=1}^S \mbe[\mathrm{SF}_{\text{MLMC}}^{(i)}(\lambda)|\theta_i]}\notag\\
&=\frac{1}{S} \mbe[\var{\mathrm{SF}_{\text{MLMC}}(\lambda)|\theta}] +\var{\frac{1}{S}\sum_{i=1}^S H(\theta_i)},\label{eq:decom}
\end{align}
where $H(\theta):=\nabla_\lambda \log q_\lambda(\theta) \left[\log p(\yobs|\theta)+ \log p(\theta)-\log q_\lambda(\theta)\right]$, $\theta_{\{1:S\}}=\{\theta_1,\dots,\theta_S\}$ and $\var{\cdot}$ and $\mbe[\cdot]$ are applied component-wisely.
The second term in \cref{eq:decom} is $O(1/S)$ when the $\theta_i$'s are generated using MC, while it should be $o(1/S)$ when the  $\theta_i$'s are generated using RQMC, or even better $O(S^{-2+\epsilon})$ if $H\circ \Gamma_\lambda$ is of BVHK.
The first term in \cref{eq:decom} is $O(1/S)$ for both cases. As a result, this strategy helps to reduce the variance in the outer sampling. Buchholz and Chopin \cite{BC:2019} applied this strategy in ABC. They found that the resulting ABC estimate has a lower variance than the MC counter-part. However, the rate of convergence cannot be improved due to the first term in \cref{eq:decom}. This strategy cannot improve the rates $r$ in \Cref{thm:sfmc,thm:remc} either.

One the other hand, we can also use a two-stage RQMC strategy. In the outer samples, we use a scrambled $(t,s')$-sequence to simulate $\theta$; while in each inner simulation, we use another independent branch of scrambled $(t,s)$-sequence to sample $x_i$. This two-stage RQMC strategy helps to reduce the noise in both inner and outer simulations. In our numerical experiments, we shall compare the effects of the three ways of using RQMC in MLMC.

\section{Numerical experiments}\label{sec:num}
\subsection{Approximate Bayesian computation}
ABC method is a generic tool in likelihood-free inference provided that it is easy to generate $y\sim  p (y|\theta)$. However, ABC methods do not target the exact posterior, but an approximation to some extent. More specially, let $\mathcal{S}(\cdot):\mathbb{R}^n\to \mathbb{R}^d$ be a vector of summary statistics, and $K_h(\cdot,\cdot)$ be  a $d$-dimensional kernel density with bandwidth $h>0$.
ABC posterior density of  $\theta$ is given by
\begin{equation*}
p_{\mathrm{ABC}}(\theta|\yobs)\propto p(\theta)\pabc(\yobs|\theta),
\end{equation*}
where the intractable likelihood is given by
\begin{equation}
\pabc(\theta|\yobs)=\int K_h(\mathcal{S}(y),\mathcal{S}(\yobs))p(y|\theta) \mrd y=\mathbb{E}_{p(y|\theta)}[K_h(\mathcal{S}(y),\mathcal{S}(\yobs))].\label{eq:abc}
\end{equation}
To fit the form \cref{eq:intrlik}, one gets  $f(y;\yobs):=K_h(\mathcal{S}(y),\mathcal{S}(\yobs))$, in which the latent variable $x$ is replaced by $y$.  To ensure $f(y;\yobs)>0$, we particularly take the Gaussian kernel
\begin{equation*}
K_h(s,\sobs)=(2\pi h)^{-d/2}\exp\left\lbrace-\frac{(s-\sobs)^\top(s-\sobs)}{2h}\right\rbrace,
\end{equation*}
where $d$ denotes the dimension of the summary statistics $\mathcal{S}(y)$.
If $\mathcal{S}(\yobs)$ is a sufficient statistic, then $p_{\mathrm{ABC}}(\theta|\yobs)$ converges to the exact posterior $p (\theta |\yobs)$ as $h\to 0$. Otherwise, $p_{\mathrm{ABC}}(\theta|\yobs)$ converges to the posterior $p(\theta |\mathcal{S}(\yobs))$ as $h\to 0$, where is a gap between $p (\theta |\mathcal{S}(\yobs))$ and $p (\theta |\yobs)$.

To apply the SF method, it suffices to provide the sample-mean likelihood estimator
\begin{equation*}
\hat{p}_N(\yobs|\theta) = \frac{1}{N}\sum_{i=1}^N K_h(\mathcal{S}(y^{[i]}),\mathcal{S}(\yobs)),
\end{equation*}
where $y^{[i]}$ are iid sample of $p(y|\theta)$. To apply the RP methods, we need to find the mappings such that
$$\theta=\Gamma(\bm u;\lambda)\sim q_\lambda(\theta)\text{ and }y =\Lambda(\bm v;\theta)\sim p(y|\theta),$$
where the distributions of $\bm u,\bm v$ do not depend on $\lambda$ and $\theta$, respectively. We also require the closed forms of $\nabla_yf(y;\yobs)$, $\nabla_\theta \Lambda(\bm v;\theta)$, and $\nabla_\lambda\Gamma(\bm u;\lambda)$. Note that
\begin{align*}
\frac{\partial f(y;\yobs)}{\partial y_i} &= \sum_{j=1}^d \frac{\partial K_h(\mathcal{S}(y),\mathcal{S}(\yobs))}{\partial \mathcal{S}_j}  \frac{\partial \mathcal{S}_j(y)}{\partial y_i}\\
&=\frac{K_h(\mathcal{S}(y),\mathcal{S}(\yobs)}{h}\sum_{j=1}^d  [\mathcal{S}_j(\yobs)-\mathcal{S}_j(y)]\frac{\partial \mathcal{S}_j(y)}{\partial y_i}.
\end{align*}
As a result, $$\nabla_yf(y;\yobs) =\frac{K_h(\mathcal{S}(y),\mathcal{S}(\yobs))\nabla_y\mathcal{S}(y)[\mathcal{S}(\yobs)-\mathcal{S}(y)]}{h}.$$
It reduces to verify and then compute the Jacobian matrix  $\nabla_y\mathcal{S}(y)$. If we take the entire data as the summary statistics, then  $\nabla_y\mathcal{S}(y)$ is an identity matrix.  If the summary statistics $\mathcal{S}(y)$ are sample moments,  $\nabla_y\mathcal{S}(y)$ can be easily computed. However, if the summary statistics $\mathcal{S}(y)$ are functions of sample quantiles, $\nabla_y\mathcal{S}(y)$ does not exist. So the SF method has a wider scope than the RP method.

\subsubsection{A toy example}
To show the unbiasedness of our methods visually, we consider a toy example of ABC which is investigated in \cite{ong:2018}. Let the data $y_1,\dots,y_n$ be from a Gaussian distribution with unknown mean $\theta$ and unit variance. We assume further the prior of $\theta$ is a standard normal distribution $N(0,1)$. Under this setting, the posterior distribution is tractable actually, which is $\theta|\yobs\sim N(n/(1+n)\bar{y}^*,1/(1+n))$, where $\bar{y}^*$ is the sample mean, but we still approximate the posterior distribution by VB methods for comparisons. Naturally, we take variational distribution $q(\theta)$ to be a normal $N(\mu,\sigma^2)$.

We take the entire data set $\yobs$ as the summary statistics (i.e., $\mathcal{S}(y)=y$) to compare the VBIL, VBSL and MLMC methods. The distribution of the summary statistic is normal, and so VBSL renders an unbiased estimator acting as a benchmark.
With the  Gaussian kernel, the ABC likelihood \cref{eq:abc} can be calculated analytically actually, which gives a guidance to choose a proper $h$. The details have been stated in \cite{ong:2018}. We take $h=0.1$ for the kernel function $K_h$ to guarantee the accuracy of the kernel approximation to the true posterior.

We test the SF and RP methods under the MC framework, respectively. In the all simulations, we consider $d=n=4$ and set the number of outer samples $S=100$ and the number of inner samples $N=100$ for all of the methods. We set the learning rate $\rho_t=1/(5+t)$. And $\alpha=1.3$ is taken for the SF methods while $\alpha=1.1$ is taken for the RP methods. We initialize the starting points for $q(\theta)$ to be $N(\bar{y}^*,1)$ and $\yobs=(0,\dots,0)$.

Figure \ref{fig:NLM} illustrates the variational posterior approximations of $\theta$ and corresponding ELBOs of VBSL, VBIL and unbiased MLMC method under the SF and RP frameworks respectively. Observed from the left panel of Figure \ref{fig:NLM}, the estimated densities of the MLMC methods and the benchmark method (VBSL) overlap considerably. On the contrary, the VBIL methods yield inaccurate densities and lower ELBOs. The ELBO of unbiased MLMC methods has more volatility than the other methods. A possible explanation is that, although the MLMC method eliminates bias, it may introduce more randomness. Nevertheless, it is apparent that MLMC methods find better variational parameters which benefit from the unbiasedness of the gradient estimators.
\begin{figure}[htbp]
  \centering
    \caption{Comparison of VBIL, VBSL and unbiased MLMC.\label{fig:NLM}}
	\subfigure[posterior distribution of $\theta$]{\includegraphics[width=4.5cm]{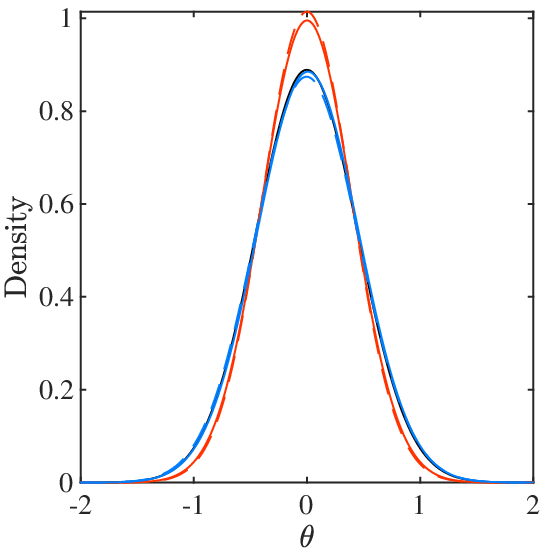}}
    \subfigure[ELBO]{\includegraphics[width=7.5cm]{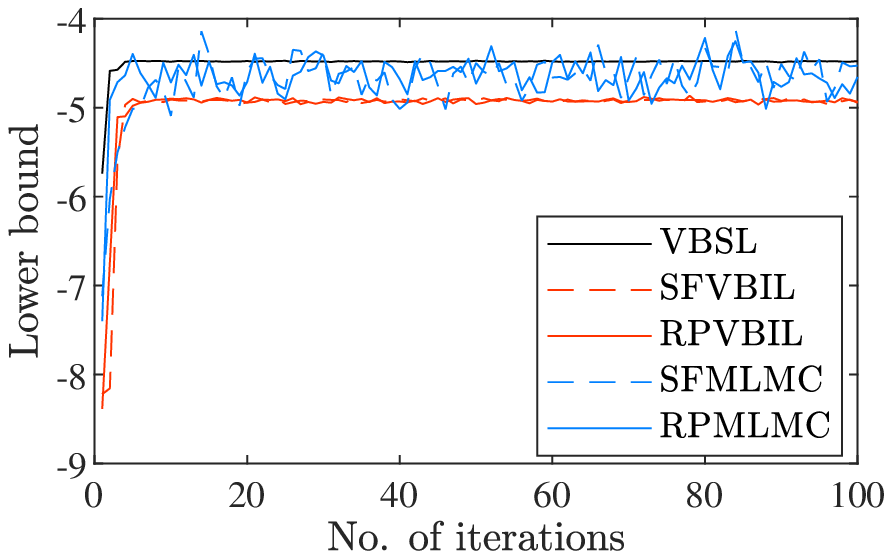}}
\end{figure}

\subsubsection{The g-and-k model}

The univariate $g$-and-$k$ distribution is a flexible unimodal distribution that
is able to describe data with significant amounts of skewness and kurtosis \cite{RM:2002}. Its density function has no closed form, but
is alternatively defined through its quantile function as:
$$Q(q|\theta)=A+B\left[1+0.8\frac{1-\exp\{-gz(q)\}}{1+\exp\{-gz(q)\}}\right](1+z(q)^2)^kz(q),$$
where $\theta=(A,B,g,k)$, $B>0, k>-1/2$, and $z(q)=\Phi^{-1}(q)$ denotes the inverse CDF of $N(0,1)$. If $g=k=0$, it reduces to  a normal distribution.  As shown in \cite{AKM:2009}, ABC is a good candidate for handling this model.

Suppose that the observations $\yobs$ of length $T=1000$ are independently generated from the $g$-and-$k$ distribution with parameter $\theta_0=(3,1,2,0.5)$. We use the unconstrained parameter
$\tilde{\theta}=(A,\log B,g,\log(k+1/2))$ in the VB and take the prior density  for $\tilde{\theta}$ as $N(0,4\cdot I_4)$. As suggested in \cite{DP:2011}, we take the summary statistics  $\mathcal{S}(y)=(\mathcal{S}_A,\mathcal{S}_B,\mathcal{S}_g,\mathcal{S}_k)$ with
\begin{align*}
\mathcal{S}_A &= E_4,\\
\mathcal{S}_B &= E_6-E_2,\\
\mathcal{S}_g &= (E_6+E_2-2E_4)/S_B,\\
\mathcal{S}_k &= (E_7-E_5+E_3-E_1)/S_B,
\end{align*}
where $E_1\le E_2\le\dots\le E_7$ are the octiles of $y$.
Note that $\mathcal{S}(y)$ is not differentiable, and thus the RP method cannot be applied. The observed summary statistics $\mathcal{S}(\yobs)=(3.05,1.63,1.58,0.42)$.

We compare MLMC and VBIL for a large bandwidth ($h=5$) and a small bandwidth ($h=0.5$), and look at the effect of bandwidth. The benchmark is the ABC acceptance-rejection (ABC-AR) samples of  size $10^4$. When $h=5$, the acceptance rate of ABC sampling is about $18\%$, while $h=0.5$, the acceptance rate reduces to $1\%$. We take $\alpha=1.3$ when $h=5$ while $\alpha=1.1$ when $h=0.5$ for the minor $h$ has effect on the smoothness of the inner function.

\begin{figure}[htbp]
  \centering
  \caption{Comparison of marginal posterior distributions.\label{fig:g-kdistribution}}
	\subfigure[$h=5$]{\includegraphics[width=13cm]{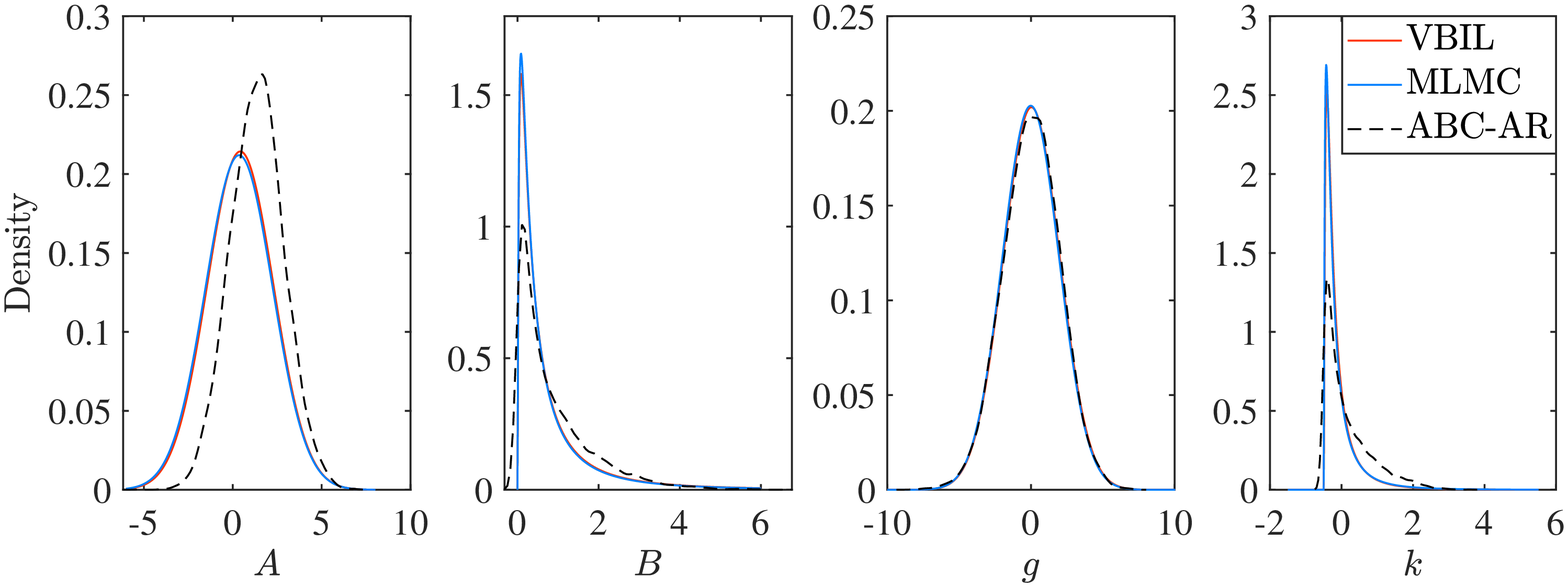}}
	\subfigure[$h=0.5$]{\includegraphics[width=13cm]{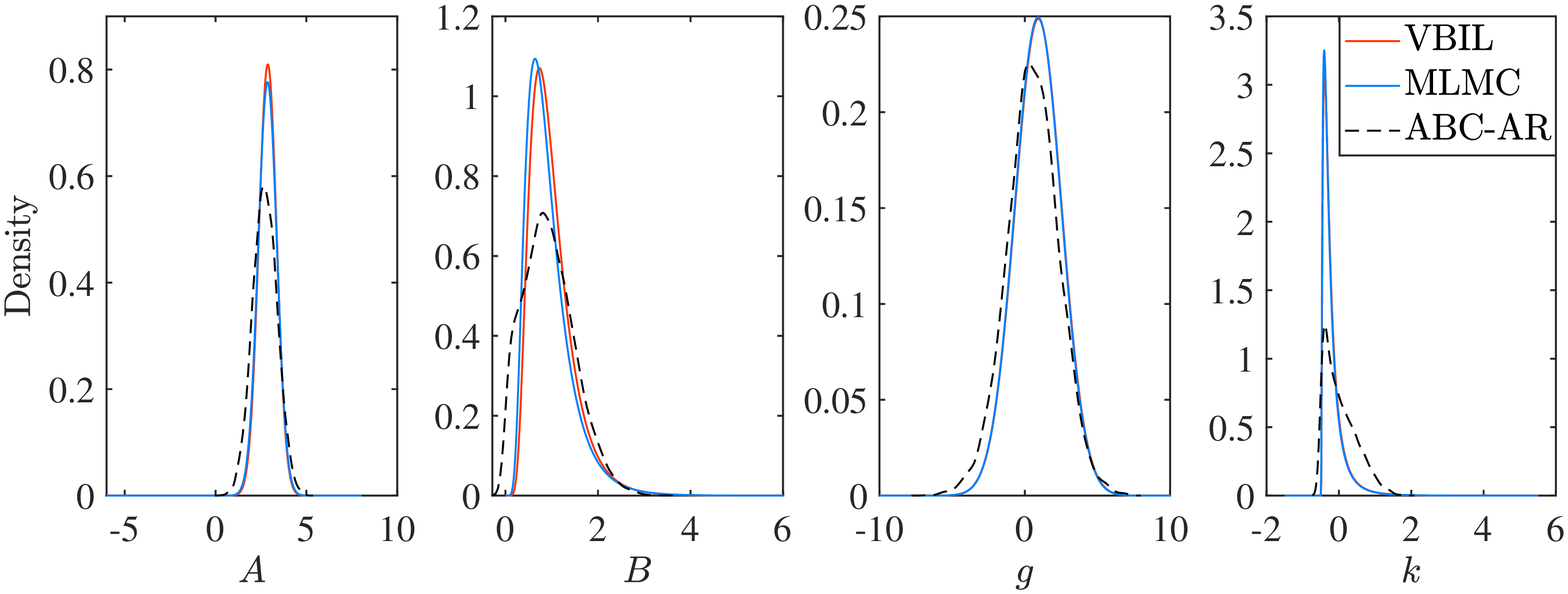}}
\end{figure}

\Cref{fig:g-kdistribution} shows the variational posterior distributions of VBIL and unbiased MLMC. As we can see, unbiased MLMC-based VB approximates the ABC posterior well, particularly for the marginal distributions of $A$ and $g$. Again, as shown in \Cref{fig:g-kLB}, unbiased MLMC leads to a larger ELBO.

\begin{figure}[htbp]
  \centering
  \caption{Comparison of ELBOs.  \label{fig:g-kLB}}
	\subfigure[$h=5$]{\includegraphics[width=6.2cm]{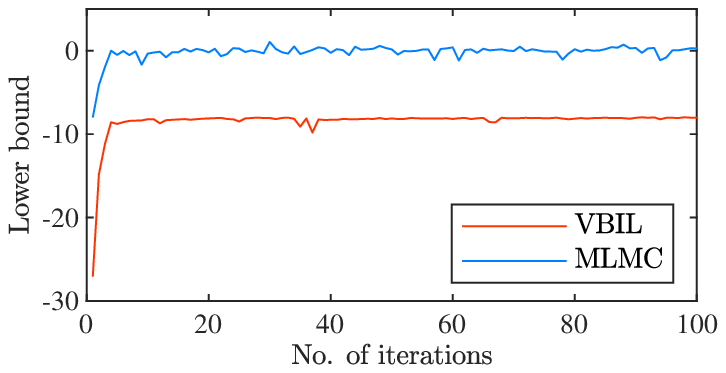}}
	\subfigure[$h=0.5$]{\includegraphics[width=6.2cm]{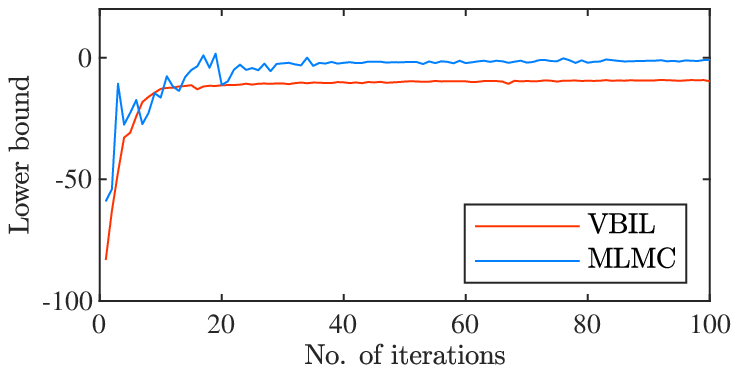}}
\end{figure}

Using RQMC in MLMC is minor for this example (the results are similar to  \Cref{fig:g-kdistribution,fig:g-kLB}, and are thus omitted for saving space). The reason is two-fold. First, it is required $1000$-dimensional RQMC points in the inner simulation, which is quite large. On the other hand, the summary statistics are functions of sampling quantiles, which are not smooth enough. Due to the high-dimensionality and the absence of smoothness in the integrands, RQMC may not perform well as expected. To overcome this, one may design some dimension reduction techniques for handling the integrand in \cref{eq:abc}.

\subsection{Generalized linear mixed models}
Generalized linear mixed models (GLMM) use a vector of random effects $\alpha_i$ to account for the dependence between the observations $y_i=\{y_{ij},j=1,\dots,n_i\}$ which are measured on the same individual $i$. The joint likelihood function of the model parameters $\theta$ and the random effects $\alpha=(\alpha_1,\dots,\alpha_n)$ is $p(\yobs,\alpha|\theta)=\prod_{i=1}^n p(\alpha_i|\theta)p(y_i|\theta,\alpha_i)$ which is tractable. However, the likelihood function $p(\yobs|\theta)=\prod_{i=1}^n p(y_i|\theta)$ with
$$p(y_i|\theta)=\int p(y_i|\theta,\alpha_i)p(\alpha_i|\theta)\mrd \alpha_i$$
is analytically intractable in most cases, while it can be easily estimated unbiasedly with importance sampling. Suppose $h_i(\alpha_i|\yobs,\theta)$ is an importance density for $\alpha_i$, then the likelihood $p(y_i|\theta)$ is estimated unbiasedly by
$$\hat p_{N_i}(y_i|\theta)=\frac{1}{N_i}
\sum_{j=1}^{N_i}\frac{p(y_i|\alpha_i^{(j)},\theta)p(\alpha_i^{(j)}|\theta)}{h_i(\alpha_i^{(j)}|\yobs,\theta)},$$
with $\alpha_i^{(j)}\stackrel{iid}{\sim} h_i(\cdot|\yobs,\theta)$.

We now compare the VBIL method and the unbiased MLMC methods using the Six City data in \cite{Fitz:1993}. The data consist of binary responses $y_{ij}$ which is the wheezing status (1 if wheezing, 0 if not wheezing) of the $i$th child at time-point $j$, where $i=1,\dots,537$ which represent 537 children and $j=1,2,3,4$ which denote $7,8,9,10$ year-old centered at 9 years correspondingly. Covariates are $A_{ij}, $the age of the $i$th child at time-point $j$ and $S_i$ the $i$th maternal smoking status (0 or 1). We consider the logistic regression model with a random intercept $y_{ij}|\beta,\alpha\sim \text{Binomial}(1,p_{ij})$, where $\text{logit}(p_{ij})=\beta_1+\beta_2A_{ij}+\beta_3S_i+\alpha_i$ with $\alpha_i\sim N(0,\tau^2)$. The parameters of this model are $\theta=(\beta,\tau^2)$. Then the likelihood function is given by
$$p(\yobs|\theta)=\prod_{i=1}^{537}\int\prod_{j=1}^4 \frac{\exp\{y_{ij}(\beta_1+\beta_2A_{ij}+\beta_3S_i+\alpha_i)\}}{1+\exp\{\beta_1+\beta_2A_{ij}+\beta_3S_i+\alpha_i\}}
\cdot \frac{1}{\sqrt{2\pi\tau^2}}\exp\{-\frac{\alpha_i^2}{2\tau^2}\}\mrd\alpha_i.$$

A normal prior $N(0,50I_3)$ is taken for $\beta$ with a $\text{Gamma}(1,0.1)$ prior for $\tau$, the square root of $\tau^2$.
We set the variational distribution $q_\lambda(\theta)$ to be a 4-dimensional normal $N(\mu,\Sigma)$, where we let $(\theta_1,\theta_2,\theta_3,\theta_4)$ denote $(\beta_1,\beta_2,\beta_3,\log\tau^2)$, which means the variational distribution of $\beta$ is a 3-dimensional normal distribution and $\tau^2$ is a log-normal distribution. This example was also investigated in \cite{tran:2017}. We focus on the RP method in this example because there is overwhelming empirical evidence in the literature showing the superiority of RP than SF. Some theoretical explanation can be found in \cite{Xu:2019}.

%Under the SF framework, the inverse of covariance matrix is used in parameterizations. When posterior variance of a certain parameter is minor, the results under this parameterizing method will be sensitive. Under this situation, SF method is not a good choice any more. As the result, we focus on RP method in this example.

In the RP method, we take $\theta=(\beta,\log\tau^2)=\mu+L\bm u$, where $\bm u\sim N(0,I_4)$.
In the inner simulation, we take $x_i=(x_{i1},\dots,x_{i4})=(z_{i1},\dots,z_{i4})+\sqrt{\tau^2}\bm v_i\cdot1_4$, where $z_{ij}=\beta_1+\beta_2A_{ij}+\beta_3S_i$, $\bm v_i\sim N(0,1)$ and $1_4$ denotes the vector $(1,1,1,1)$.

Firstly, we test the decreasing rates of $\mathbb{E}[\|\nabla_\lambda \Gamma(\bm u;\lambda)\Delta \tilde \psi_{\theta,\ell}\|_2^2]$ for testing  MLMC-based gradient estimation and $\mathbb{E}[|\Delta \psi_{\theta,\ell}|^2]$ for testing  MLMC-based  ELBO estimation. We run the algorithms starting with $\mu=(0,0,0,0)^T,\Sigma=I_4$ and $M_0=16$. We compare the cases of using MC and RQMC in the inner simulation. To get accurate estimates of these quantities, we use RQMC in the outer sampling. As shown in \Cref{fig:norm}, we find that $r=1.52$ for the gradient estimator when RQMC is used in the inner simulation while $r=1.43$ for MC in the inner. Also, RQMC leads to a larger $r=1.96$ for the ELBO estimator. When MC is used in the inner, we take $\alpha=1.4$ to finalize the probability distribution of $w_\ell$. While $\alpha=1.5$ when RQMC is used in the inner.  A large $\alpha$ speeds up the VB algorithm. According to \cref{eq:cost}, RQMC reduces the cost by a factor of $16\%$ compared to MC.

\begin{figure}[htbp]
  \centering
  \caption{Tests of the decrease rates.\label{fig:norm}}
	\subfigure[Gradient of ELBO]{\includegraphics[width=4.2cm]{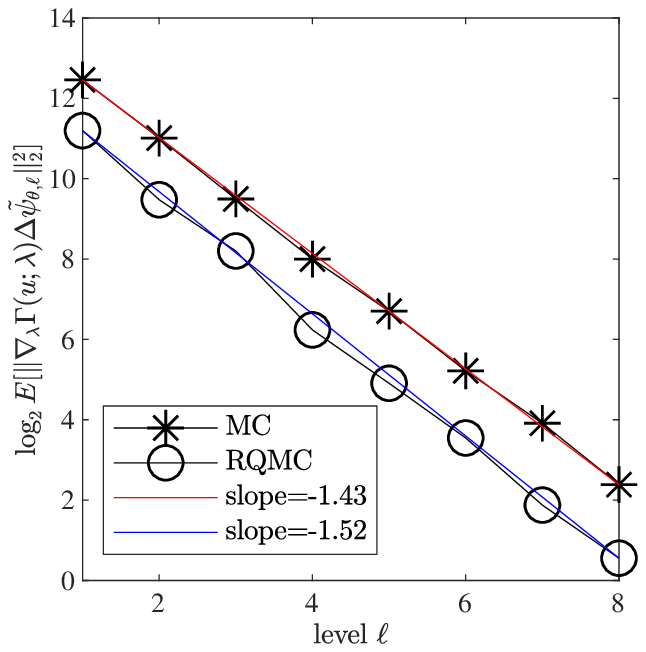}}
    \subfigure[ELBO]{\includegraphics[width=4.2cm]{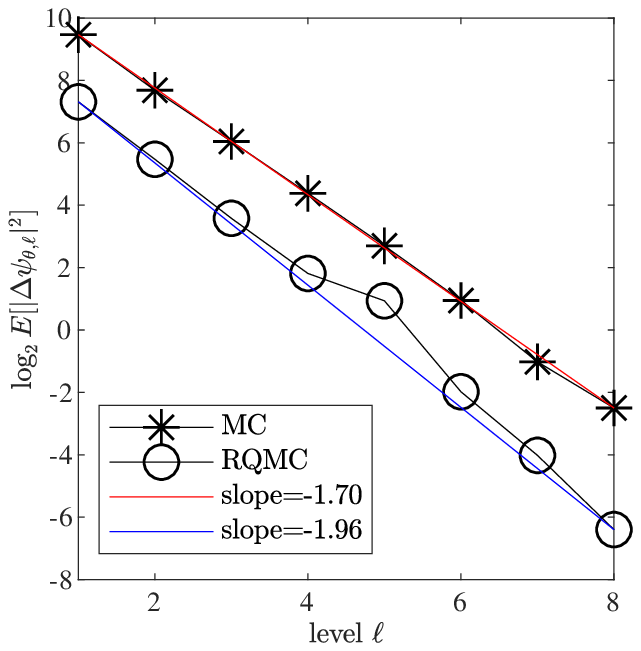}}
  \end{figure}

\begin{table}[htbp]
    \scriptsize
	\centering
	\caption{Variances of unbiased MLMC-based gradient estimators for the initial variational parameters. `I' is short for `Inner', `O' for `Outer', `M' for `MC' and `Q' for `RQMC'. }
	\label{tab:variance}
	\begin{tabular}{ccccc ccccc ccccc}
		\toprule
		 I/O&$\beta_1$&$\beta_2$&$\beta_3$&$\tau^2$&$L_{11}$&$L_{21}$&$L_{31}$&$L_{41}$&$L_{22}$&$L_{23}$&$L_{24}$&$L_{33}$&$L_{34}$&$L_{44}$\\\hline
		M/M&152&164&30&41&253&182&54&118&226&28&74&55&99&32 \\
		M/Q&69&97&11&30&171&142&26&99&215&18&89&29&87&30 \\
       Q/M&111&84&17&22&260&148&34&49&162&31&47&27&33&17 \\
		Q/Q&82&80&13&40&170&146&30&93&161&24&86&20&30&11\\
	\bottomrule
	\end{tabular}
\end{table}

The results in \Cref{fig:norm} show that RQMC can improve the sampling accuracy in the inner simulation with a large $r$, but the effect of RQMC used in the outer simulation is still unclear. To this end, we estimate the variance of the unbiased MLMC-based gradient estimator for the initial variational parameters by $50$ repetitions. The empirical variances are shown in \Cref{tab:variance}. It can be seen that using RQMC in either inner or outer simulation reduce the variances for most parameters. Variance reduction of gradient estimates should help to improve VB.

\begin{figure}[htbp]
  \centering
  \caption{Comparison of VBIL and four unbiased MLMC methods: MC+MC, MC+RQMC, RQMC+MC and RQMC+RQMC.\label{fig:RPmethod}}
    \subfigure[Marginal posterior distributions]{\includegraphics[width=13cm]{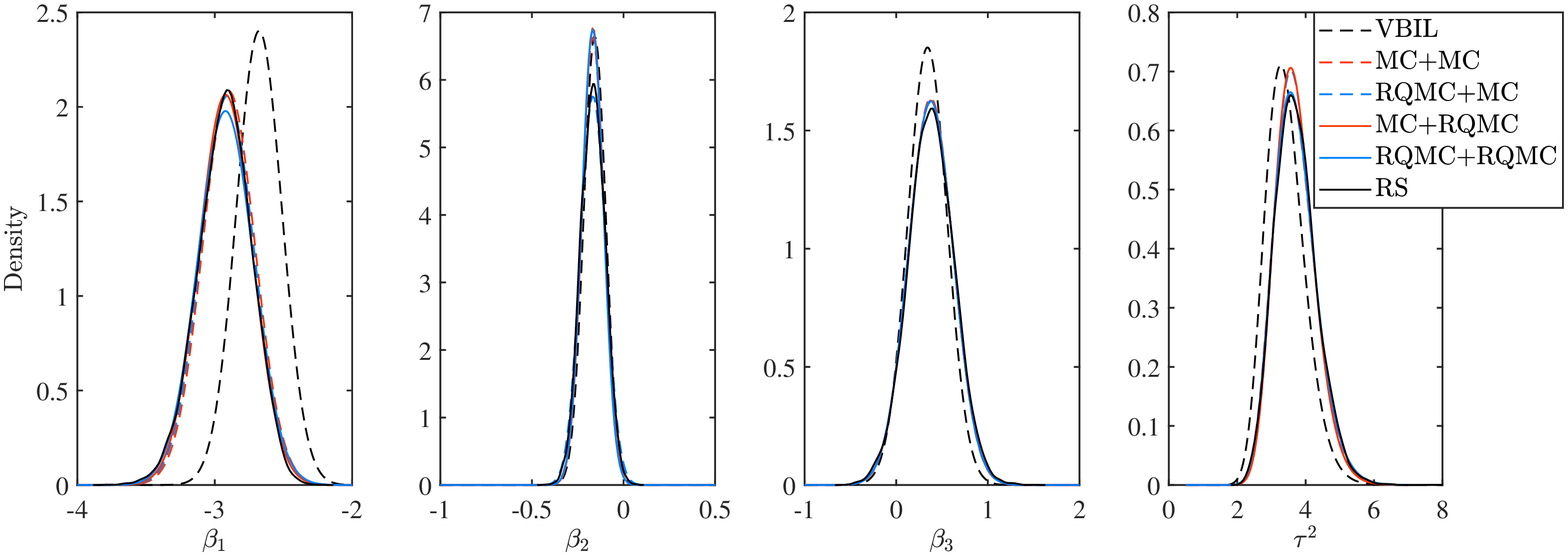}}
    \subfigure[ELBO]{\includegraphics[width=13cm]{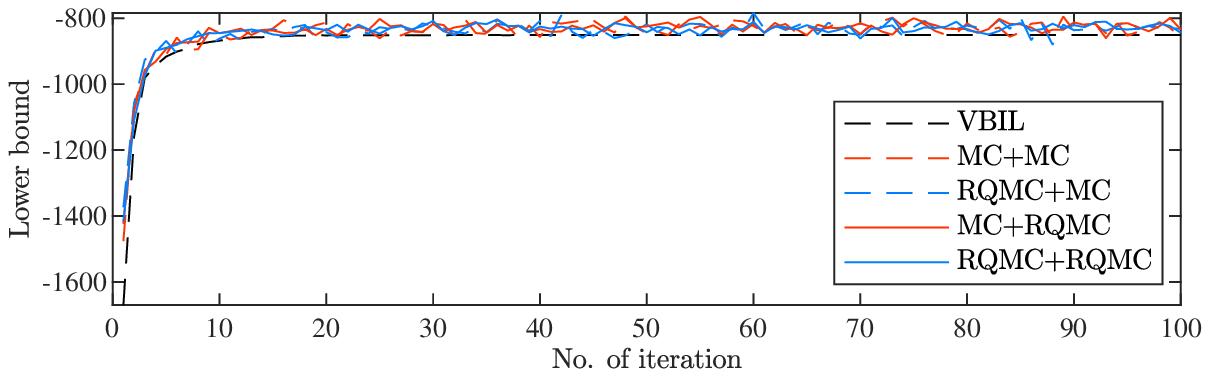}}
\end{figure}

Finally, we compare VBIL with four unbiased MLMC methods: MC+MC, MC+RQMC, RQMC+MC and RQMC+RQMC, where for example, MC+RQMC means the MC method is used in the outer while the RQMC method is used in the inner and so on. We take $M_0=8$ for the unbiased MLMC methods and $N=16$ for VBIL. The RStan package `rstanarm' is used to sample from $p(\theta|\yobs)$ as a benchmark, which performs posterior analysis for
models with dependent data such as GLMMs. As shown in Figure \ref{fig:RPmethod}, unbiased MLMC-based methods show great consistency with the benchmark distribution (labeled as RS). On the other hand, all unbiased MLMC methods lead to larger ELBOs than VBIL.

\section{Concluding remarks}\label{sec:concl}
In this paper, we developed a general method to deal with VB problems with intractable likelihoods. The central point is to find an unbiased gradient estimator in stochastic gradient-based optimization. We achieve this goal by designing unbiased nested MLMC estimators for both the SF and RP gradients. Compared to VBIL, our proposed methods find a better fitting of the posterior distribution and a tighter estimate of the marginal likelihood. Compared to VBSL, our methods work with general distributions of summary statistics. To improve the sampling efficiency, we incorporated RQMC in the inner and the outer simulations. Using RQMC in the inner simulation can reduce the average cost of unbiased MLMC. Using RQMC in the outer simulation can reduce the variance of the gradient estimator.
Both aspects speed up the VB algorithm.

%\section*{Acknowledgments}

%\bibliographystyle{siamplain}
%\bibliography{references}

\end{document}